\documentclass[aop,MSNbibl,seceqn,dvips]{arximspdf}
\usepackage{graphicx}
\doi{10.1214/13-AOP832} 
\volume{42}
\issue{2}
\pubyear{2014}
\firstpage{431}
\lastpage{463}

\makeatletter
\newcommand{\iint}{\int\!\!\!\int}
\newcommand{\eqref}[1]{(\ref{#1})}
\newtheorem{proposition}{Proposition}[section]
\newtheorem{theorem}[proposition]{Theorem}
\newtheorem{corollary}[proposition]{Corollary}
\newtheorem{lemma}[proposition]{Lemma}
\newproclaim{remark}[proposition]{Remark}
\newproclaim{definition}[proposition]{Definition}
\newproclaim{example}[proposition]{Example}
\newproclaim{question}{Question}

\newcommand{\I}{\mathbf{1}}
\newcommand{\nc}{\newcommand}
\newcommand{\eps}{\varepsilon}
\nc{\R}{{\mathbb R}}
\nc{\N}{{\mathbb N}}
\nc{\Z}{{\mathbb Z}}
\nc{\id}{\operatorname{id}}
\nc{\card}{\operatorname{card}}
\nc{\supp}{\operatorname{supp}}
\nc{\BP}{\mathbb{P}}
\nc{\BQ}{\mathbb{Q}}
\nc{\BE}{\mathbb{E}}
\makeatother

\begin{document}
\begin{frontmatter}

\title{Unbiased shifts of Brownian motion\thanksref{T1}}
\runtitle{Unbiased shifts of Brownian motion}
\thankstext{T1}{Supported by a grant from the {Royal Society}.}

\begin{aug}
\author[A]{\fnms{G\"unter} \snm{Last}\corref{}\ead[label=e1]{guenter.last@kit.edu}},
\author[B]{\fnms{Peter} \snm{M\"orters}\ead[label=e2]{maspm@bath.ac.uk}}
\and
\author[C]{\fnms{Hermann} \snm{Thorisson}\ead[label=e3]{hermann@hi.is}}
\runauthor{G. Last, P. M\"orters and H. Thorisson}
\affiliation{Karlsruhe Institute of Technology, University of Bath
and\break
University of Iceland}
\address[A]{Institut f\"ur Stochastik\\
Karlsruhe Institute of Technology\\
Kaiserstra{\ss}e 89\\
D-76128 Karlsruhe\\
Germany\\
\printead{e1}} 
\address[B]{Department of Mathematical Sciences\\
University of Bath\\
Claverton Down, Bath BA2 7AY\\
United Kingdom\\
\printead{e2}}
\address[C]{Science Institute\\
University of Iceland\\
Dunhaga 3, 107 Reykjavik\\
Iceland\\
\printead{e3}}
\end{aug}

\received{\smonth{12} \syear{2011}}
\revised{\smonth{9} \syear{2012}}

%
\begin{abstract}
Let $B=(B_t)_{t\in\R}$ be a two-sided standard Brownian motion.
An \emph{unbiased shift}~of~$B$ is a random time~$T$, which
is a measurable function of~$B$, such that $(B_{T+t}-B_T)_{t\in\R}$
is a Brownian motion independent of~$B_T$. We characterise unbiased shifts
in terms of allocation rules balancing
mixtures of local times of~$B$.
For any probability distribution~$\nu$ on $\R$ we construct a stopping
time $T\ge0$ with the above properties such that $B_T$ has
distribution~$\nu$.
We also study moment
and minimality properties of unbiased~shifts.
A crucial ingredient of our approach is a new theorem
on the existence of allocation rules
balancing stationary diffuse random measures on~$\R$.
Another new result is an analogue for diffuse random
measures on $\R$ of the cycle-stationarity characterisation of
Palm versions of stationary simple point processes.
\end{abstract}

%

\end{frontmatter}

\section{Introduction and main results}\label{sintro}

Let $B=(B_t)_{t\in\R}$ be a two-sided standard Brownian motion in $\R$
having $B_0=0$. If $T\ge0$ is a stopping time with respect
to the filtration $(\sigma\{B_s \dvtx s\le t\})_{t\ge0}$,
then the shifted process $(B_{T+t}-B_T)_{t\ge0}$ is a one-sided Brownian
motion independent of $B_T$. However, the two-sided shifted process
$(B_{T+t}-B_T)_{t\in\R}$ need not be
a two-sided Brownian motion. Moreover, the example of a
fixed time shows that even if it is, it need not
be independent of $B_T$. We call a random time~$T$
an \emph{unbiased shift} (of a two-sided Brownian motion) if
$T$ is a measurable function of $B$, and $(B_{T+t}-B_T)_{t\in\R}$
is a two-sided Brownian motion, independent of $B_T$.
We say that a random time $T$ \emph{embeds} a given
probability measure $\nu$ on $\R$, often called the \emph{target distribution},
if $B_T$ has distribution~$\nu$. 

In this paper we discuss several
examples of nonnegative unbiased shifts that are stopping times.
However, we wish to stress that nonnegative
unbiased shifts are not assumed to have the stopping time property;
see, for instance, Example~\ref{exnonstop}.
The paper has three main aims. The \emph{first} aim is to characterise
all unbiased shifts that embed a given
distribution~$\nu$. The \emph{second} aim is to construct such unbiased shifts.
In particular, we solve the Skorokhod embedding
problem for unbiased shifts: given
any target distribution we find an unbiased shift which embeds this
target distribution
(and is also a stopping time).
The \emph{third} and final aim is to discuss 
properties of unbiased shifts. In particular, we discuss optimality of
our solution of
the Skorokhod embedding problem for unbiased~shifts.

The case when the embedded distribution
is concentrated at zero is of special interest.
Let $\ell^0$ be the local time at zero.
Its right-continuous (generalised) inverse is
defined by
%
%
\begin{equation}
\label{Tr} T_r:= \cases{ \sup\bigl\{t \geq0 \dvtx \ell^0
[0,t] = r\bigr\}, &\quad $r \geq0,$\vspace*{2pt}
\cr
\sup\bigl\{t < 0 \dvtx
\ell^0 [t,0] = -r\bigr\}, &\quad $r < 0.$}
\end{equation}
Note that $\BP_0\{T_0= 0\}=1$
and $\BP_0\{T_r= 0\}=0$ if $r\ne0$. We prove the following theorem.

\begin{theorem}\label{littletwin}
Let $r\in\R$. Then $T_r$ is an unbiased shift embedding $\delta_0$.
\end{theorem}

This result formalises the intuitive idea that two-sided
Brownian motion looks globally the same from all its
(appropriately chosen) zeros,
thus resolving an issue raised by Mandelbrot in \cite{Ma82}, pages 207 and
385,
and reinforced in \cite{KaVere96,Thor00}.
Another way of thinking about this result is that if we travel in time
according to the clock of local time,
we always see a two-sided Brownian motion.

The property described in Theorem~\ref{littletwin} is analogous to a
well-known feature
of the two-sided stationary Poisson process
with an extra point at the origin: the lengths of the intervals between points
are i.i.d. (exponential) and therefore shifting the origin
to the $n$th point on the right (or on the left)
gives us back a two-sided Poisson process with
an extra point at the origin.
In the Poisson case the process with an extra point at the origin
is the Palm version of the stationary process, and
it is a well-known characterising property
of Palm versions of stationary point processes on the line
that their distributions do not change when the origin is shifted along
the points.

In fact, much of the work behind the present paper was
inspired and motivated by the recent literature on matching and
allocation problems.
There is a strong analogy between the problem of finding an extra head
in a two-sided sequence of independent fair coin tosses, as discussed
in~\cite{L02},
and the problem of finding an unbiased shift for Brownian motion embedding
a given probability distribution. Unlocking this analogy was key to the
solution of
the latter problem. But the analogy extends further to the more recent
developments for spatial point processes and random
measures~\cite{HP05,HPPS09,LaTh09}. In the terminology of \cite{LaTh09},
Theorem~\ref{littletwin} means that Brownian motion is
\emph{mass-stationary} with respect to local time; see Section~\ref{secmassstat}
below. Holroyd and Peres~\cite{HP05} consider the balancing of
Lebesgue measure and a
stationary ergodic spatial point process, obtaining the Palm version of
the point process by
shifting the origin to the associated point of the process.
Last and Thorisson~\cite{LaTh09} extend these ideas to
the balancing of general random measures
in an abstract group setting.
This general theory, and Poisson-matching ideas from
\cite{HPPS09}, are essential for the present paper
where we consider the balancing of local times at different levels.

Theorem \ref{littletwin} is relatively elementary.
To state the further main results of this paper we now briefly
introduce some notation and terminology, full details for our framework
will be given in Section~\ref{secPalm}. To begin with, it is convenient to
define $B$ as the identity on the canonical probability space
$(\Omega,\mathcal{A},\BP_0)$, where $\Omega$ is the set
of all continuous functions $\omega\dvtx\R\rightarrow\R$,
$\mathcal{A}$ is the Kolmogorov product $\sigma$-algebra,
and $\BP_0$ is the distribution of $B$. Define
$\BP_x:=\BP_0\{B+x\in\cdot\}$, $x\in\R$, and
the $\sigma$-finite and stationary measure
%
%
\begin{equation}
\label{statP} \BP:=\int\BP_x \,dx.
\end{equation}
%
Expectations (resp., integrals) with respect to $\BP_x$ and $\BP$ are denoted
by $\BE_x$ and $\BE_\BP$, respectively.
For any $t\in\R$ the shift $\theta_t\dvtx\Omega\rightarrow\Omega$ is
defined by
%
%
\begin{equation}
\label{shiftB} (\theta_t\omega)_s:=
\omega_{t+s}.
\end{equation}
An \emph{allocation rule} \cite{HP05,LaTh09} is a measurable mapping
$\tau\dvtx\Omega\times\R\rightarrow\R$ that is \emph{equivariant}
in the sense that
%
%
\begin{equation}
\label{allocation} \tau(\theta_t\omega,s-t)=\tau(\omega,s)-t, \qquad s,t\in
\R, \mbox{$\BP$-a.e. $\omega\in\Omega$}.
\end{equation}
A \emph{random measure} $\xi$ on $\R$ is a kernel
from $\Omega$ to $\R$ such that $\xi(\omega,C)<\infty$
for $\BP$-a.e. $\omega$ and all compact $C\subset\R$.
If $\xi$ and $\eta$ are random measures, and
$\tau$ is an allocation rule such that the image measure
of $\xi$ under $\tau$ is $\eta$, that is,
%
%
\begin{equation}
\label{xe} \int\I\bigl\{\tau(s)\in\cdot\bigr\} \xi(ds)=\eta,\qquad \BP\mbox{-a.e.},
\end{equation}
then we say that $\tau$ \emph{balances} $\xi$ and $\eta$.
If $\tau$ balances $\xi$ and $\eta$, and $\sigma$ is
an allocation rule that balances $\eta$ and another
random measure $\zeta$, then the allocation rule
$\sigma\circ\tau$ balances $\eta$ and $\zeta$.
Let $\ell^x$ be the random measure
associated with the \emph{local time} of $B$
at $x\in\R$.
For a locally finite measure $\nu$ on $\R$ we define
%
%
\begin{equation}
\label{ellnu} \ell^\nu(\omega,\cdot):=\int\ell^x(\omega,
\cdot) \nu(dx), \qquad \omega\in \Omega.
\end{equation}
%
Since $\ell^x$ is supported by $\{t\in\R\dvtx B_t=x\}$
and $B$ is bounded on bounded intervals,
we obtain
that $\ell^\nu$ is
$\BP$-a.e. finite on bounded sets and hence a random measure.
The random measure $\ell^\nu$ has the invariance property
\[
\ell^\nu(\theta_t\omega,C-t)=\ell^\nu(
\omega,C), \qquad C\in\mathcal{B},t\in\R, \BP\mbox{-a.s.}
\]
For any random time $T$ we define an allocation rule
$\tau_T$ by
%
%
\begin{equation}
\label{tauT} \tau_T(t):=T\circ\theta_t+t,\qquad t\in\R.
\end{equation}
Since $T=\tau_T(0)$, there is a
one-to-one correspondence
between $T$ and $\tau_T$. Let us emphasise again that Section~\ref{secPalm} will provide
further details regarding the notation introduced in this paragraph.

Our 
key characterisation
theorem is based on a result
in \cite{LaTh09}, which will be recalled as Theorem~\ref{th2} below.

\begin{theorem}\label{main1} Let $T$ be a random time
and $\nu$ be a probability measure on~$\R$. Then
$T$ is an unbiased shift embedding $\nu$ if and
only if $\tau_T$ balances $\ell^0$ and $\ell^\nu$.
\end{theorem}

For any probability measure $\nu$ on $\R$
we denote by $\BP_\nu:=\int\BP_x \nu(dx)$
the distribution of a two-sided Brownian motion with
a random starting value $B_0$ with law~$\nu$.
We show in Section~\ref{secmassstat} that all these
distributions coincide on the invariant $\sigma$-algebra.
A general result in \cite{Thor96} (see also \cite{Kallenberg}, Theorem 10.28)
then implies that there is a random time~$T$,
possibly defined on an extension of $(\Omega,\mathcal{A}, \BP_0)$,
such that $\theta_T B$ has distribution $\BP_\nu$ under~$\BP_0$.
The next two theorems yield a much stronger result.
They show that $T$ can be chosen as a \emph{factor}
of $B$, that is, as a measurable function of $B$;
see \cite{HP05} for a similar result for Poisson processes.
Moreover, this factor is explicitly known. The
proof
is based on Theorem \ref{main1} and on a general
result on the existence of allocation rules
balancing
stationary orthogonal diffuse random measures
on $\R$ with equal conditional intensities; see Theorem~\ref{main3}
below.

\begin{theorem}\label{main2}
Let $\nu$ be a probability measure on $\R$ with $\nu\{0\}=0$.
Then the stopping time
%
%
\begin{equation}
\label{bertoin} 
T^\nu:=\inf \bigl\{t> 0\dvtx
\ell^0[0,t]=\ell^\nu[0,t] \bigr\}
\end{equation}
embeds $\nu$ and is an unbiased shift.
\end{theorem}

The stopping time
$T^\nu$
was introduced in
\cite{BertJan93} as a solution of the \emph{Skorokhod embedding problem}.
This problem requires finding a stopping time
$T\ge0$ embedding a given distribution $\nu$; see
\cite{Ob04} for a survey.
The idea of using mixtures of local times to solve
this problem was introduced in \cite{monroe72b}.
It has apparently not been noticed
before that
$T^\nu$
is an unbiased shift.
The
methods of the present paper are
very different from the methods of~\cite{monroe72b,BertJan93}.

It is important to note that in the two-sided framework being a
stopping time
is neither necessary nor sufficient for the shifted process to be a Brownian
motion. For instance, this is not the case for another stopping time
introduced in \cite{BertJan93}, which is defined similar to~\eqref
{bertoin} but with
$\nu$ replaced by a finite measure of mass exceeding one; see
Remark~\ref{rsuper2}.
Conversely, unbiased shifts need not be stopping times, even when they
are nonnegative;
see Example~\ref{exnonstop}.\vadjust{\goodbreak}

If $\nu$ is of the form $\nu\{0\}\delta_0+(1-\nu\{0\})\mu$
where $\mu\{0\}=0$ and $\nu\{0\}>0$, then Theorem~\ref{main2}
does not apply. In fact, if $\nu\{0\}<1$, then
$T^\nu$
is an unbiased
shift embedding~$\mu$. Still we can use Theorem~\ref{main2}
to construct unbiased shifts without any assumptions
on $\nu$:

\begin{theorem}\label{main2a}
Let $\nu$ be a probability measure on $\R$.
Then there exists a nonnegative stopping time that
is an unbiased shift embedding $\nu$.
\end{theorem}

In Theorem~\ref{littletwin} we have $\BP_0\{B_{T_0}=0,T_0= 0\}=1$
and $\BP_0\{B_{T_r} =0, T_r \neq 0\}=1$ if $r\ne0$.
It is interesting to note that unbiased shifts $T$ (even if they are
not stopping times)
are almost surely nonzero as long
as the condition $\BP_0\{B_T=0\}<1$ is fulfilled:

\begin{theorem}\label{prop3} Let $\nu$ be a probability
measure on $\R$ such that $\nu\{0\}<1$. Then any unbiased shift $T$
embedding $\nu$ satisfies
%
%
\begin{equation}
\label{21} \BP_0\{T=0\}=0.
\end{equation}
\end{theorem}

In contrast to the previous
theorem, if $T$ is an unbiased shift with\break
\mbox{$\BP_0\{B_T=0\}=1$}, then the probability  $\BP_0\{T=0\}$ may take
any value:

\begin{theorem}\label{p11} For any $p\in[0,1]$ there is an unbiased
shift $T\ge0$ embedding $\delta_0$ and such that $\BP_0\{T=0\}=p$.
\end{theorem}

A solution $T$ of the
Skorokhod embedding problem is usually required to have
good moment properties, but some restrictions apply. For instance,
if the target distribution $\nu$ is not centered,
by~\cite{MortersPeres}, Theorem 2.50,
we must have $\BE_0\sqrt{T}=\infty$. If the embedding stopping time
is also an unbiased shift the situation is worse, even
when $\nu$ is centred.

\begin{theorem}\label{t14complete}
Suppose $\nu$ is a target distribution with $\nu\{0\}=0$,
and the stopping time $T\geq0$ is an
unbiased shift embedding~$\nu$. Then
\[
\BE_0 T^{1/4}=\infty.
\]
If $\nu$ additionally satisfies
$\int|x| \nu(dx)<\infty$, the unbiased shift
$T=T^\nu$
satisfies
%
\[
\BE_0 T^{\beta}<\infty\qquad\mbox{for all $\beta<1/4$.}
\]
\end{theorem}

Dropping the stopping time
assumption, we show in Theorem~\ref{pmoment} that
$\BE_0\sqrt{|T|}=\infty$ for any unbiased shift $T$ embedding a
target distribution $\nu$ with $\nu\{0\}<1$.
If the target distribution is concentrated
at zero, and $T$ is nonnegative but not identically zero,
we show in Theorem~\ref{posmoment} that
$\BE_0 T=\infty$. Nonnegativity is important in this result;
Example~\ref{exexpon} provides an unbiased shift
with $\BP_0\{T\ne0\}=1$ that has exponential moments.

Theorem \ref{pminimal} further shows that,
in addition to the nearly optimal moment properties stated above,
the stopping times
$T^\nu$
defined in~\eqref{bertoin} are also \emph{minimal}
in a sense analogous to the definition in \cite{monroe72}; see also
\cite{CH}, or \cite{Ob04} for a survey.
This means that if $S\ge0$ is another unbiased shift
embedding $\nu$ such that $\BP_0\{S\le T\}=1$, then $\BP_0\{S=T\}=1$.
Our discussion of minimality is based on a notion of stability of
allocation rules, which is similar to the one studied in \cite{HPPS09}.

The results for Brownian motion stated above will be developed in a
general framework,
which goes much beyond the Brownian setting; see Sections~\ref{secmassstat} and~\ref{secexistence}. They
are heavily reliant on
the general Palm theory from~\cite{LaTh09}. The most important results,
which are also
of independent interest, are Theorem~\ref{mass-stat}, characterising
mass-stationarity,
and Theorem~\ref{main3}, formulating general conditions on the
existence of balancing
allocation rules.

The structure of the paper is as follows. Section~\ref{secPalm}
presents essential
background on Palm measures and local time.
Section~\ref{secmassstat}
establishes a general result on mass-stationarity
for diffuse random measures on the line, 
Theorem~\ref{mass-stat},
implying
a result (Theorem~\ref{littletwingen}) containing Theorem~\ref
{littletwin} as a special case.
Section~\ref{sexinv} proves a result (Theorem~\ref{main1a}) containing
Theorem~\ref{main1} as a special case.
Section~\ref{secexistence} presents
the key general result on
balancing diffuse 
random measures,
Theorem \ref{main3},
implying
a result (Theorem~\ref{main2general}) containing Theorem~\ref{main2} as
a special case.
Section~\ref{seczero}
proves
Theorems~\ref{main2a}, \ref{prop3}
and \ref{p11}.
In Sections~\ref{secstability} and \ref{secproperties} we establish
minimality and moment properties of unbiased shifts,
including Theorem~\ref{t14complete}.
Section~\ref{secremarks} finally
discusses extensions of the central results above
from Brownian motion to
a more general class of L\'evy processes.

\section{Preliminaries on Palm measures and local times}\label{secPalm}

In order to present and develop some Palm theory on which the
results of this paper rely, 
we need a framework more general
than the Brownian setting in the \hyperref[sintro]{Introduction}.
Consider a $\sigma$-finite measure space
$(\Omega,\mathcal{A},\BP)$ equipped with a \emph{flow}
$\{\theta_t\dvtx t\in\R\}$ of measurable bijections
$\theta_t\dvtx\Omega\rightarrow\Omega$ such that
$(\omega,t)\mapsto\theta_t(\omega)$ is measurable,
$\theta_0$ is the identity on~$\Omega$ and
$\theta_{s+t}=\theta_s\circ\theta_t$ for all $s,t\in\R$.
We assume that $\BP$ is \emph{stationary}, that is,
%
%
\begin{equation}
\label{Pinv} \BP=\BP\circ\theta_s,\qquad s\in\R.
\end{equation}
By the \emph{stationary Brownian case} we mean the important example when
$\Omega$ is the class of all continuous functions $\omega\dvtx\R
\rightarrow\R$
with the flow given by \eqref{shiftB},
$\mathcal{A}$ is the Kolmogorov product $\sigma$-algebra,
and the measure $\BP$ is
given by \eqref{statP}.
We use the term \emph{Brownian case} when the stationary $\BP$
is (possibly) replaced by other Brownian measures like $\BP_0$ and $\BP_x$.
In the Brownian case we let $B=(B_t)_{t\in\R}$ denote the
identity on $\Omega$.
Since $\BP\{B_0\in C\}<\infty$ for any compact $C\subset\R$,
the measure $\BP$ is indeed $\sigma$-finite, and the
proof of \eqref{Pinv} is
based
on the stationary increments of $B$; see~\cite{Zaehle91}.
Corollary~\ref{timeP} 
below
provides an alternative definition of $\BP$.
More general
L\'evy processes will be discussed in Section~\ref{secremarks}.

Random measures and (balancing) allocation rules are
defined as in Section~\ref{sintro}.
A random measure $\xi$ is called \emph{invariant} if
%
%
\begin{equation}
\label{xicov} \xi(\theta_t\omega,C-t)=\xi(\omega,C),\qquad C\in
\mathcal{B},t\in\R, \BP\mbox{-a.s.},
\end{equation}
where $\mathcal{B}$ is the Borel $\sigma$-algebra on $\R$.
In this case the \emph{Palm measure} $\BQ_\xi$ of $\xi$
(with respect to $\BP$) is defined by
%
%
\begin{equation}
\label{defPalm} \BQ_\xi(A):=\BE_\BP\int
\I_{[0,1]}(s)\I_A(\theta_s) \xi(ds),\qquad A\in
\mathcal{A}.
\end{equation}
This is a $\sigma$-finite measure on $(\Omega,\mathcal{A})$.
If the \emph{intensity} $\BQ_\xi(\Omega)$ of $\xi$ is
positive and finite, $\BQ_\xi$ can be normalised to
yield the Palm probability measure of $\xi$.
This measure can be interpreted as the conditional distribution
(with respect to $\BP$) given that the origin $0\in\R$
is a \emph{typical point} in the mass of $\xi$;
see \cite{Kallenberg}, Chapter~11, for some fundamental properties
of Palm probability measures.
The invariance property \eqref{xicov}
implies the \emph{refined Campbell theorem}
%
%
\begin{equation}
\label{refC} \BE_\BP\int f(\theta_s,s) \xi(ds)=
\BE_{\BQ_\xi} \int f(\theta_0,s) \,ds
\end{equation}
for any measurable $f\dvtx\Omega\times\R\rightarrow[0,\infty)$
where, as in~\eqref{defPalm}, $\BE_{\BQ}$ denotes integration with
respect to
the measure~$\BQ$. The relevance of Palm measures for this paper stems
from the following result in~\cite{LaTh09}.

\begin{theorem}\label{th2}
Consider two invariant random measures $\xi$ and $\eta$ on $\R$ and
an allocation rule $\tau$. Then $\tau$ balances $\xi$ and $\eta$
if and only if
\[
\BQ_\xi\{\theta_{\tau}B\in\cdot\}=\BQ_\eta,
\]
where $\theta_{\tau}\dvtx\Omega\to\Omega$ is defined by
$\theta_{\tau}(\omega):=\theta_{{\tau}(\omega,0)}\omega$, $\omega\in
\Omega$.
\end{theorem}

In the remainder of this section we consider the Brownian case.
Recall that $\ell^x$ is the random measure
associated with the \emph{local time} of $B$
at $x\in\R$ (under $\BP_0$). This means that
%
%
\begin{equation}
\label{occupation} \int f(B_s,s) \,ds=\iint f(x,s)
\ell^x(ds) \,dx, \qquad\BP_0\mbox{-a.s.}
\end{equation}
for all measurable $f\dvtx\R^2\rightarrow[0,\infty)$.
The global construction in \cite{Perkins81}
(see also \cite{Kallenberg}, Proposition 22.12
and \cite{MortersPeres}, Theorem~6.43)
guarantees the existence of a version of local times
with the following properties.
The random measure $\ell^0$ is $\BP_x$-a.e. diffuse for
any $x\in\R$ and
%
%
\begin{eqnarray}
\label{0inv} \ell^0(\theta_t\omega,C-t)&=&
\ell^0(\omega,C), \qquad C\in\mathcal{B}, t\in\R, \BP_x
\mbox{-a.s.}, x\in\R,
\\
\label{0x} \ell^y(\omega,\cdot)&=&\ell^0(\omega-y,\cdot),\qquad
\omega\in\Omega, y\in\R,
\\
\label{0y} \int\I\{B_t\ne x\} \ell^x(\omega, dt) &=& 0,\qquad
\omega\in\Omega, x\in\R.
\end{eqnarray}
%
Equation \eqref{0x} implies that $\ell^y$ is $\BP_x$-a.e. diffuse for
any $x\in\R$ and is invariant in the sense of \eqref{0inv}.
From Fubini's theorem we infer that these properties
do also hold for the random measure $\ell^\nu$ defined by \eqref{ellnu}.

%
\begin{remark}\label{addfunc}
Invariant random measures of the form \eqref{ellnu} are closely related
to \emph{continuous additive functionals} of Brownian motion;
see, for example, \cite{Kallenberg}, Chapter~22.
Indeed, if $\xi$ is an invariant random measure,
then the process $A_t:=\xi[0,t]$, $t\ge0$, is additive in the sense
that $A_{s+t}=A_s+A_t\circ\theta_s$ for all $s,t\ge0$
($\BP$-a.s.). Conversely, if $(A_t)_{t\ge0}$ is additive
and continuous ($\BP_x$-a.s. for all $x\in\R$), and if $A_t$ depends only
on the restriction of $B$ to the interval $[0,t]$,
then \cite{Kallenberg}, Chapter~22, implies that there is
a locally finite measure $\nu$, the \emph{Revuz measure} of $(A_t)$, such
that $(A_t)_{t\ge0}=(\ell^\nu[0,t])_{t\ge0}$ $\BP_x$-a.s. for all
$x\in\R$.
\end{remark}

The following result is essentially from
\cite{GeHo73}; see also \cite{Zaehle91}. Combined with Theorem \ref{th2}
it will yield a short proof of Theorem \ref{main1}; see Section~\ref{sexinv}.

\begin{lemma}\label{Palmlocal} Let $y\in\R$. Then
$\BP_y$ is the Palm measure of $\ell^y$.
\end{lemma}
\begin{pf}
Let $f\dvtx\Omega\times\R\rightarrow[0,\infty)$
be measurable. By definition \eqref{statP} of $\BP$
we have
\[
\BE_\BP\int f(\theta_sB,s) \ell^y(B,ds) =
\int\BE_0\int f(\theta_sB+x,s) \ell^y(B+x,ds)
\,dx.
\]
By \eqref{0x} this equals
\[
\int\BE_0\int f(\theta_sB+x,s) \ell^{y-x}(B,ds)
\,dx =\BE_0\iint f(\theta_sB-x+y,s) \ell^{x}(B,ds)
\,dx,
\]
where the equality comes from a change of variables
and Fubini's theorem. By~\eqref{occupation} this equals
\[
\BE_0\int f(\theta_sB-B_s+y,s) \,ds.
\]
Since $\BP_0\{\theta_sB-B_s+y\in\cdot\}=\BP_y$,
we obtain
%
%
\begin{equation}
\label{refCy} \BE_\BP\int f(\theta_sB,s)
\ell^y(B,ds) =\BE_y \int f(B,s) \,ds
\end{equation}
and hence the assertion.
\end{pf}

Equation~\eqref{refCy} is the refined Campbell theorem
\eqref{refC} in the case $\xi=\ell^y$.
In particular, it implies that $\ell^y$ has intensity $1$,
%
%
\begin{equation}
\label{int1} \BE_\BP\ell^y\bigl([0,1]\bigr)=1.
\end{equation}

\section{Mass-stationarity}\label{secmassstat}

In this section 
we show that the
property
in Theorem~\ref{littletwin} characterises
mass-stationarity (defined below) not only of
local times of Brownian motion but
of general diffuse random measures on the line.
As in Section~\ref{secPalm}
we consider a measurable space $(\Omega,\mathcal{A})$,
equipped with a flow $\{\theta_t\dvtx t\in\R\}$.
We consider a $\sigma$-finite measure $\BQ$ on
$(\Omega,\mathcal{A})$ but do not assume that $\BQ$
is stationary. The key example in the Brownian case is $\BQ=\BP_x$.

Let $\xi$ be a diffuse and invariant random measure on $\R$
[\eqref{xicov} is assumed to hold $\BQ$-almost everywhere],
and let $\lambda$ denote Lebesgue measure on $\R$.
Then $\xi$ is called \emph{mass-stationary} if,
for all bounded Borel subsets $C$ of $\R$ with $\lambda(C)>0$
and $\lambda(\partial C)=0$
and all measurable functions $f\dvtx\Omega\times\R\rightarrow[0,\infty)$,
%
%
\begin{equation}
\label{mass-stat-def}\qquad \BE_\BQ\iint\I_C(u)
\frac{\I_{C-u}(s)}{\xi(C-u)} f(\theta_s, s + u) \xi(ds) \,du = \BE_\BQ
\int\I_C(u) f(\theta_0, u) \,du,
\end{equation}
using the convention that any integration over a set of measure zero
yields zero.
Mass-stationarity is a formalisation
of the intuitive idea that the origin is
a typical location in the mass of a random measure.

Property \eqref{mass-stat-def}
can be interpreted probabilistically as saying that, if the set $C$ is
placed uniformly at random around the origin and the origin shifted to a
location chosen according to the mass distribution of $\xi$ in
this
randomly placed set, then
the distribution of $\xi$ does not change.

The
property \eqref{gen-inv} in the following theorem
is a new characterisation of mass-stationarity.
It is similar to the well-known characterisation
by cycle-stationarity
in the simple point process case (see, e.g., \cite{Kallenberg}, Theorem 11.4)
and is certainly more transparent than \eqref{mass-stat-def}.
It is, however, restricted to the diffuse case on the line
while \eqref{mass-stat-def} works for general random measures
in a group setting.
The
formula~\eqref{Q*} below is also new, but the
equivalence of mass-stationarity and Palm measure
property
was established in \cite{LaTh09} for Abelian groups
and in \cite{Last2010} for general locally compact groups.

\begin{theorem}\label{mass-stat} Assume that
$\BQ\{\xi(-\infty,0)<\infty\}=\BQ\{\xi(0,\infty)<\infty\}=0$, and
let $S_r$, $r \in\R$, be the
generalised inverse of the diffuse random measure~$\xi$
defined as in~\eqref{Tr}. Then
%
%
\begin{equation}
\label{gen-inv} 
\BQ \{
\theta_{S_r}\in\cdot \}=\BQ,\qquad r\in\R,
\end{equation}
if and only if $\xi$
is mass-stationary and if and only if $\BQ$
is the Palm measure of $\xi$ with respect to
a $\sigma$-finite stationary measure $\BP$.
The measure $\BP$ is uniquely determined by $\BQ$ as follows:
for each $w>0$ and each measurable function
$f\dvtx\Omega\rightarrow[0,\infty)$,
%
%
\begin{equation}
\label{Q*} \BE_{\BP} f=w^{-1}\BE_{\BQ}\int
_0^{S_w} f\circ\theta_s \,ds.
\end{equation}
\end{theorem}
\begin{pf}
First assume \eqref{gen-inv}. Then
$\BQ\{\xi[0,\varepsilon]=0\}=\BQ\{\xi[S_1,S_1+\varepsilon]=0\}=0$
for any $\varepsilon>0$, where the second identity comes from
$\xi[S_1,\infty)>0$ $\BQ$-a.e. and the definition of $S_1$. This easily
implies that
%
%
\begin{equation}
\label{keyo} S_r = -S_{-r}\circ\theta_{S_r},\qquad \BQ
\mbox{-a.e.}, r\in\R.
\end{equation}

Let $C\subset\R$ be a bounded Borel with $\lambda(C)>0$
and $\lambda(\partial C)=0$. Changing variables
and noting that, for any $s$ in the support of $\xi$, we have
$\xi(C-v+s)>0$ for $\lambda$-a.e. $v\in C$,
we obtain that the left-hand side of \eqref{mass-stat-def} equals
\begin{eqnarray*}
&&\BE_\BQ\iint\I_C(v-s) \frac{\I_{C}(v)}{\xi(C-v + s)}f(
\theta_s,v) \xi (ds) \,dv
\\
&&\qquad=\BE_\BQ\iint\I_C(v-S_r)
\frac{\I_{C}(v)}{\xi(C-v + S_r)}f(\theta _{S_r},v) \,dr \,dv,
\end{eqnarray*}
where we have changed variables to get the equality.
The key observation \eqref{keyo}
and assumption \eqref{gen-inv} yield that the above equals
\begin{eqnarray*}
&&\BE_\BQ\iint\I_C(v+S_{-r})
\frac{\I_{C}(v)}{\xi(C-v)}f(\theta_0, v) \,dr \,dv
\\
&&\qquad=\BE_\BQ\iint\I_C(v+s) \frac{\I_{C}(v)}{\xi(C-v)}f(
\theta_0,v) \xi(ds) \,dv =\BE_\BQ\int \I_{C}(v)f(
\theta_0,v) \,dv. 
\end{eqnarray*}
Thus \eqref{mass-stat-def} holds, that is, $\xi$ is mass-stationary.

By \cite{LaTh09}, Theorem 6.3, equation \eqref{mass-stat-def}
is equivalent to the existence of
a stationary \mbox{$\sigma$-finite} measure $\BP$ such that
$\BQ$ is the Palm measure of $\xi$ with respect to $\BP$.
Mecke's \cite{Mecke} inversion formula
(see also \cite{LaTh09}, Section~2) implies that
$\BP$ is uniquely determined by $\BQ$ and
that, moreover,
$\BP\{\xi(-\infty,0)<\infty\}=\BP\{\xi(0,\infty)<\infty\}=0$.

Fix $w>0$. For the claim that
$\BP$ defined by \eqref{Q*} is stationary when \eqref{gen-inv} holds,
see Lemma~\ref{secondQ*} below.
To show that $\BQ$ is then the Palm measure of $\xi$
with respect to this $\BP$, let $f\dvtx\Omega\rightarrow[0,\infty)$ be
measurable
and use \eqref{Q*} for the first step in the following calculation:
\begin{eqnarray*}
w \BE_{\BP} \int\I_{[0,1]}(s) f\circ\theta_s
\xi(ds) &= &\BE_{\BQ}\iint \I_{[0,1]}(s) \I_{[0,S_w]}(t) f(
\theta_s\theta_t) \theta_t\xi(ds) \,dt
\\
&=&\BE_{\BQ}\iint \I_{[0,1]}(v-t) \I_{[0,S_w]}(t) f(
\theta_v) \xi(dv) \,dt
\\
&=& \BE_{\BQ}\iint \I_{[0,1]}(S_r-t)
\I_{[0,S_w]}(t) f(\theta_{S_r}) \,dr \,dt
\\
&=& \BE_{\BQ}\iint \I_{[0,1]}(-S_{-r}-t)
\I_{[0, S_{w}(\theta
_{S_{-r}})]}(t) f \,dr \,dt
\\
&=& \BE_{\BQ}f\iint \I_{[-1,0]}(u) \I_{[S_{-r}, S_{-r+w}]}(u) \,dr \,du
\\
&=& w \BE_{\BQ}f,
\end{eqnarray*}
where we have used \eqref{gen-inv} and \eqref{keyo}
for the fourth identity, and
the final identity holds since the double integral equals $w$.

Finally, if $\BQ$ is the Palm measure of $\xi$ with respect to
a $\sigma$-finite stationary measure~$\BP$,
then Theorem \ref{th2} implies
\eqref{gen-inv} once we have shown for any $r\in\R$ that
the allocation rule $\tau^r$ defined by
$\tau^r(t):=S_r\circ\theta_t+t$ balances $\xi$ with itself,
that is,
\[
\int\I\bigl\{\tau^r(s)\in\cdot\bigr\} \xi(ds)=\xi, \qquad\BP\mbox{-a.e.}
\]
Assume $r\ge0$. Then, outside the $\BP$-null set
$A:=\{\xi(-\infty,0)<\infty)\}\cup\{\xi(0,\infty)<\infty\}$
we obtain for any $a<b$
(interpreting $\xi[s,a]$ as $-\xi[a,s]$ for $s\ge a$)
that
%
%
\begin{eqnarray}
\label{867}
&&\int\I\bigl\{a\le\tau^r(s)< b\bigr\} \xi(ds)\nonumber \\
&&\qquad=
\int\I\bigl\{s\le b,\xi[s,a]\le r, \xi[s,b]> r\bigr\} \xi(ds)
\\
&&\qquad= \int\I\bigl\{s\le b,r < \xi[s,b]\le r+\xi[a,b] \bigr\} \xi(ds)=
\xi[a,b],\nonumber
\end{eqnarray}
which implies the desired balancing property.
The case $r<0$ can be treated similarly.
\end{pf}

\begin{lemma}\label{secondQ*}
Let $S\ge0$ be a random time and $\BP$ be the measure defined
by setting, for each measurable function $f\dvtx\Omega\rightarrow
[0,\infty)$,
\[
\BE_{\BP} f= \BE_{\BQ}\int_0^{S}
f\circ\theta_s \,ds.
\]
If $\BQ\{\theta_S\in\cdot\}=\BQ$, then $\xi$ is stationary under~$\BP$.
\end{lemma}
\begin{pf}
For each
$f$ as above and $t\in\R$,
\begin{eqnarray*}
\BE_{\BP} f(\theta_t)& =& \BE_{\BQ} \int
_t^{S+t} f(\theta_s) \,ds
\\
& =& \BE_{\BQ} \int_t^{S} f(
\theta_s) \,ds + \BE_{\BQ}\int_S^{S+t}
f(\theta_s) \,ds
\\
& =& \BE_{\BQ}\int_t^{S} f(
\theta_s) \,ds + \BE_{\BQ} \int_0^{t}
f(\theta_s) \,ds
\\
&= &\BE_{\BQ} \int_0^{S} f(
\theta_s) \,ds = \BE_{\BP} f,
\end{eqnarray*}
where the third identity follows from the assumption that
$\theta_S$ has the same distribution as $\theta_0$ under $\BQ$.
\end{pf}

In the remainder of this section we consider the Brownian case.
As a corollary
of Theorem~\ref{mass-stat}
we obtain
an alternative construction of the stationary measure~\eqref{statP} by integrating
 over time rather than space.

\begin{corollary}\label{timeP} 
Let $r>0$ and 
$T_r$ be defined by \eqref{Tr}. Then
%
\[
\BP(A):=\int\BP_x(A) \,dx = r^{-1}
\BE_0 \int^{T_r}_0\I\{
\theta_sB\in A\} \,ds,\qquad A\in\mathcal A.
\]
\end{corollary}

Generalising our earlier definition,
for any probability measure~$\mu$ on $\R$, we call a random time $T$ an
\emph{unbiased shift under~$\BP_\mu$} if
$(B_{T+t}-B_T)_{t\in\R}$ under $\BP_\mu$ is
a Brownian motion independent of~$B_T$.
The following result 
contains Theorem~\ref{littletwin} as a special case.

\begin{theorem}\label{littletwingen} 
Let $\mu$
be a probability measure on $\R$, and let $S_r$, $r \in\R$, be the
generalised~inverse of
$\ell^\mu$ defined as in~\eqref{Tr}. Then each $S_r$ is an unbiased shift
under $\BP_\mu$ and $\BP_\mu\{B_{S_r}\in\cdot\}=\mu$.
\end{theorem}
\begin{pf}
Lemma \ref{Palmlocal} and Fubini's theorem imply
that $\BP_\mu$ is the Palm measure of $\ell^\mu$ with
respect to $\BP$. Hence the result follows from Theorem \ref
{mass-stat}.
\end{pf}

The \emph{invariant $\sigma$-algebra} is defined
by
%
%
\begin{equation}
\label{invsigma} \mathcal I:=\{A\in\mathcal{A}\dvtx\mbox{$\theta_t
A=A$ for all $t\in\R $}\}.
\end{equation}
We now apply Theorem \ref{littletwin} to prove the following result
which we need in the proof of Theorem \ref{main2} in Section~\ref{secexistence}.

%
\begin{theorem}\label{H1} 
Let $A\in\mathcal I$. Then either $\BP_x(A)=0$
for all $x\in\R$ [in which case $\BP(A)=0$]
or $\BP_x(A^c)=0$ for all $x\in\R$ [in which case $\BP(A^c)=0$].
\end{theorem}
\begin{pf} We first show that
%
%
\begin{equation}
\label{01} \BP_0(A)\in\{0,1\}.
\end{equation}
We use here the random times $T_n$ [see \eqref{Tr}]
for integers $n$.
By Theorem \ref{littletwin}, for any integer $n$, the processes
$(B_{T_n-t})_{t\ge0}$ and $(B_{T_n+t})_{t\ge0}$ are
independent one-sided Brownian motions. This implies that the
processes
\[
W_n:=(B_{(T_n+t)\wedge(T_{n+1}-T_n)})_{t\ge0},
\]
are independent under $\BP_0$.
Since, by \eqref{0inv},
\[
\inf\bigl\{t\ge0 \dvtx\ell^0\bigl(\theta_{T_n}B,[0,t]
\bigr)=1\bigr\} =\inf\bigl\{t\ge0 \dvtx\ell^0\bigl(B,[T_n,T_n+t]
\bigr)=1\bigr\} =T_{n+1}-T_n
\]
holds $\BP_0$-a.s. for any $n\in\Z$,
the $W_n$ have the distribution
of a one-sided Brownian motion stopped at the time
its local time at $0$ reaches the value $1$.
Clearly we have that
$B=F((W_n)_{n\in\Z})$ for a suitably defined measurable function $F$.
By invariance of $A$ and definition of the family $(W_n)_{n\in\Z}$,
\[
\bigl\{F\bigl((W_{n})_{n\in\Z}\bigr)\in A\bigr\}=\{B\in A\} =\{
\theta_{T_1}B\in A\}=\bigl\{F\bigl((W_{n+1})_{n\in\Z}
\bigr)\in A\bigr\},
\]
where the final equation holds $\BP_0$-a.s.
As i.i.d. sequences are ergodic under shifts (see, e.g., Theorem~8.45
in~\cite{Dudley}),
the invariant sets above have measure zero or one, implying~\eqref{01}.

The refined Campbell theorem \eqref{refCy} implies
(with $\lambda$ denoting Lebesgue measure)
%
%
\begin{equation}
\label{345} \BP_x(A)=\lambda(C)^{-1}\BE_\BP
\I_A\ell^x(C), \qquad x\in\R,
\end{equation}
provided that $0<\lambda(C)<\infty$. Assume now that $\BP_0(A)=0$.
Then \eqref{345} implies that
\[
\BP\bigl(A\cap\bigl\{\ell^0(C)>0\bigr\}\bigr)=0
\]
for all compact $C\subset\R$. Letting $C\uparrow\R$, we obtain
$\BP(A\cap\{\ell^0\ne0\})=0$, that is,
\[
\BP_x\bigl(A\cap\bigl\{\ell^0\ne0\bigr\}\bigr)=0,\qquad \lambda
\mbox{-a.e. $x$}.
\]
On the other hand, by \eqref{0x},
$\BP_x\{\ell^0\ne0\}=\BP_0\{\ell^{-x}\ne0\}=1$
for $\lambda$-a.e. $x$
so that $\BP_x(A)=0$ for $\lambda$-a.e. $x$.
Therefore $\BP(A)=0$. By \eqref{345} this implies $\BP_x(A)=0$
for all $x\in\R$.
\end{pf}

\section{Unbiased shifts and balancing allocation rules}\label{sexinv}

In this section we consider the Brownian case and
prove the following result which
contains Theorem~\ref{main1} as a special case.
Let $\mu$ be a probability measure on $\R$, and recall
from Section~\ref{secmassstat} that a random time $T$
is an unbiased shift under $\BP_\mu$
if $(B_{T+t}-B_T)_{t\in\R}$ is under $\BP_\mu$
a Brownian motion independent of~$B_T$.

\begin{theorem}\label{main1a} Let $T$ be a random time
and $\mu,\nu$ be probability measures on~$\R$. Then
the following two assertions are equivalent:
\begin{longlist}[{(ii)}]
\item[{(i)}] $T$ is an unbiased shift under $\BP_\mu$ and
$\BP_\mu\{B_T\in\cdot\}=\nu$;
%
\item[{(ii)}] the allocation rule $\tau_T$ defined
by \eqref{tauT} balances $\ell^\mu$ and $\ell^\nu$.
\end{longlist}
\end{theorem}
\begin{pf}
First we recall from Section~\ref{secPalm}
that the random measures $\ell^\mu$ and $\ell^\nu$
are invariant in the sense of \eqref{xicov}.

Let us first assume that (i) holds. Then we have for
any $A\in\mathcal{A}$ that
\begin{eqnarray*}
\BP_\mu\{\theta_T B\in A\}&= &\int\BP_\mu\{
\theta_T B-B_T+x\in A\} \nu(dx)
\\
&=&\int\BP_0\{B+x\in A\} \nu(dx)=\BP_\nu(A).
\end{eqnarray*}
Lemma \ref{Palmlocal} and Fubini's theorem imply that
$\BP_\nu$ is the Palm measure of $\ell^\nu$. Therefore
we obtain from Theorem \ref{th2} that $\tau_T$ balances
$\ell^\mu$ and $\ell^\nu$.

Assume now that (ii) holds. By Theorem \ref{th2} we obtain
for any $A\in\mathcal{A}$ that
\[
\BP_\mu\{\theta_T B\in A\}=\int\BP_x(A)
\nu(dx).
\]
This implies
\[
\BP_\mu\bigl\{\theta_T B-B_T\in
A',B_T\in C\bigr\}= \int_C
\BP_x\bigl\{B-x\in A'\bigr\} \nu(dx) =\BP_0
\bigl(A'\bigr)\nu(C)
\]
for any $A'\in\mathcal{A}$ and any $C\in\mathcal{B}$.
This yields (i).
\end{pf}

\begin{remark}\label{remsub} An \emph{extended} allocation rule
is a mapping $\tau\dvtx\Omega\times\R\rightarrow[0,\infty]$
that has the equivariance property \eqref{allocation}.
The balancing property \eqref{xe} can then be defined
as before. Using these concepts, Theorem \ref{main1a} can be proved
for a subprobability measure $\nu\ne0$. The conditions in (i)
have to be replaced with
$\BP_\mu\{\theta_T B-B_T\in\cdot\mid T<\infty\}=\BP_0$,
$\BP_\mu\{T<\infty,B_T\in\cdot\}=\nu$ and the
independence of $\theta_T B-B_T$ and $B_T$
under $\BP_\mu\{\cdot|T<\infty\}$.
\end{remark}

\section{Existence of unbiased shifts}\label{secexistence}

In this section we prove
a result (Theorem~\ref{main2general}) containing Theorem~\ref{main2} as
a special case.
The proof is based on the following new balancing result for general
random measures on the line, which is inspired by~\cite{HPPS09}.
As in Section~\ref{secPalm} we consider a $\sigma$-finite measure space
$(\Omega,\mathcal{A},\BP)$, equipped with a flow
$\{\theta_t\dvtx t\in\R\}$ such that $\BP$ is stationary.
The invariant
$\sigma$-algebra $\mathcal{I}$ is defined as at \eqref{invsigma}.

\begin{theorem}\label{main3} Let $\xi$ and $\eta$ be
invariant orthogonal diffuse random measures on $\R$
with finite intensities. Assume further that
\[
\BE_\BP \bigl[ \xi[0,1] | \mathcal I \bigr] = \BE_\BP
\bigl[ \eta[0,1] | \mathcal I \bigr], \qquad\BP\mbox{-a.e.}
\]
Then the mapping $\tau\dvtx\Omega\times\R\rightarrow\R$, defined by
%
%
\begin{equation}
\label{matching} \tau(s):=\inf\bigl\{t> s \dvtx\xi[s,t]=\eta[s,t]\bigr\},\qquad  s\in\R,
\end{equation}
is an allocation rule balancing $\xi$ and $\eta$.
\end{theorem}


We start the proof of Theorem~\ref{main3} with an analytic lemma.
Here and later it is convenient to work with the 
continuous function $f\dvtx\R\rightarrow\R$, defined by
\begin{eqnarray*}
f(t):=\cases{ \xi[0,t]-\eta[0,t],&\quad$\mbox{if $t\ge0$,}$\vspace*{2pt}
\cr
\eta[t,0]-\xi[t,0],&\quad $\mbox{if $t< 0$}.$}
\end{eqnarray*}

\begin{lemma}\label{analytic}
Suppose $\xi$ and $\eta$ are
orthogonal diffuse measures. Then
\[
\int\I\bigl\{\tau(s)\in\cdot\bigr\} \xi(ds)=\eta(\cdot) \qquad\mbox{on $[0,a]$},
\]
provided that $f(t)\ge0$ for all $t\in(0,a)$.
\end{lemma}

The proof of Lemma~\ref{analytic} rests on three further lemmas.

\begin{lemma}\label{des}
\begin{longlist}[(a)]
\item[(a)] For $\xi$-almost every $s$ there exists $s_n\downarrow s$
with $f(s_n)>f(s)$.
\item[(b)] For $\eta$-almost every $s$ there exists $s_n\downarrow s$
with $f(s_n)<f(s)$.
\end{longlist}
\end{lemma}

\begin{pf}
It suffices to prove~(a), as (b) follows by reversing the roles of
$\xi$ and~$\eta$.
Recall that $\xi$ and $\eta$ are
orthogonal, and hence there exists
a Borel set $A$ with $\eta(A)=0$ and $\xi(A^{ c})=0$. We need to
show that,
for each $\varepsilon>0$,
\[
\xi(A_\varepsilon)=0,\qquad \mbox{where } A_\varepsilon:=\bigl\{ s\in A \dvtx
\mbox{$f(t)\le f(s)$ for all $t\in[s,s+\varepsilon)$}\bigr\}.
\]
Given any $\delta>0$ we may choose an open set
$O \supset A$ with $\eta(O)<\delta$. We can cover
$A_\varepsilon$ by a countable collection $\mathcal I$ of nonoverlapping intervals
$[s,s+\varepsilon_s]$, $s\in A_\varepsilon$, $0<\varepsilon_s\leq\varepsilon$, such that
$(s,s+\varepsilon_s)\subset O$.
Indeed, suppose that $O'$ is a connected component of $O$, which
intersects~$A_\varepsilon$.
If there is a minimal element $s$ in $O'\cap A_\varepsilon$ let $\varepsilon_s$
be the minimum of $\varepsilon$ and the distance of $s$ to the
right endpoint of $O'$. Add the interval $[s,s+\varepsilon_s]$ to
the collection $\mathcal I$ and remove it from $A_\varepsilon$ and $O$.
If no such minimum exists we can pick a strictly
decreasing sequence $s_n\in O'\cap A_\varepsilon$, $n\in\N$, converging to
the infimum.
Let $\varepsilon_{s_1}$ be the minimum of $\varepsilon$ and
the distance of $s_1$ to the right endpoint of $O'$, and, for $i\geq2$,
let $\varepsilon_{s_i}$ be the minimum
of $\varepsilon$ and $s_{i-1}-s_{i}$. Add all intervals $[s_i,s_i+\varepsilon
_{s_i}]$ to
the collection $\mathcal I$ and remove their union from $A_\varepsilon$
and $O$.
Note that after one such step (performed in every connected component)
all of $A_\varepsilon$ in connected components of length at most $\varepsilon
$ will
be removed, and the lower bound of the intersection of all other
connected components
with $A_\varepsilon$, if finite, is increased by at least~$\varepsilon$.
Also, after one step,
the intersection of any connected component with $A_\varepsilon$ is either empty
or bounded from below. Therefore,
every set of the form $[-M,M]\cap A_\varepsilon$ will
be completely covered after finitely many steps by nonoverlapping intervals,
as required. Observe that $\xi(I)\leq\eta(I)$ for every interval
in the collection, and hence
\[
\xi(A_\varepsilon)  \leq\sum_{I\in\mathcal I}\xi(I) \leq
\sum_{I\in
\mathcal I} \eta(I) \leq \eta(O) \leq\delta.
\]
The result follows as $\delta>0$ was arbitrary.
\end{pf}

\begin{figure}

\includegraphics{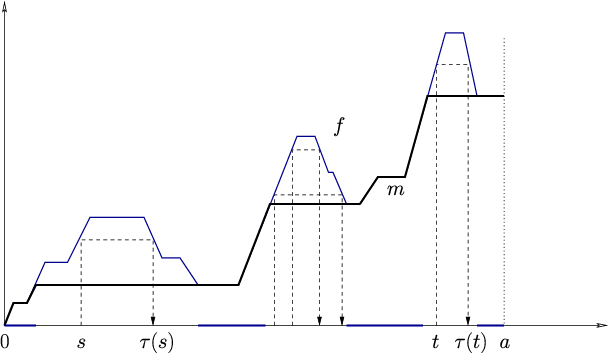}

\caption{{A schematic picture of the function $f$ and its backwards running
minimum $m$, in bold. The set $C$ is marked
bold on the abscissa, and instances of the mapping $t\mapsto\tau(t)$
are indicated by dashed lines.}}\label{fig1}
\end{figure}

We now fix $a\geq0$ and decompose $f$ on $[0,a]$
according to its backwards running minimum~$m$ given by
\[
m(t)= \min\bigl\{ f(s) \dvtx t \leq s \leq a\bigr\};
\]
see Figure~\ref{fig1} for illustration.
The nonnegative function $f-m$ can be decomposed on $[0,a]$ into a family
$\mathcal E$ of excursions
$e\dvtx[0,\infty)\to[0,\infty)$ with starting times $t_e\in[0,a]$.
Note that
an excursion $e\dvtx[0,\infty)\to[0,\infty)$ is a function such that there
exists a number $\sigma_e>0$, called the lifetime of the excursion,
such that
$e(0)=0$, $e(s)>0$ for all $0<s<\sigma_e$, and $e(s)=0$ for all $s\geq
\sigma_e$.
Formally putting $e(s)=0$ for all $s<0$ the decomposition can be
written as
\[
f(t)-m(t)=\sum_{(e,t_e)\in\mathcal E} e(t-t_e).
\]
Note that the intervals $(t_e,t_e+\sigma_e)$, $(e,t_e)\in\mathcal E$, are
disjoint. We denote by $C$ the complement of their union in $[0,a]$,
that is, $C=\{t\in[0,a] \dvtx f(t)=m(t)\}$.


\begin{lemma}
For every $(e,t_e)\in\mathcal E$ we have
\[
\xi\bigl\{s\in(t_e,t_e+\sigma_e) \dvtx
\tau(s)\leq a\bigr\}= \xi(t_e,t_e+\sigma_e)
=\eta(t_e,t_e+\sigma_e).
\]
\end{lemma}

\begin{pf}
We only have to show that $\tau(s)\leq a$ for $\xi$-almost every
$s\in(t_e,t_e+\sigma_e)$. By Lemma~\ref{des}(a),
for $\xi$-almost every $s\in(t_e,t_e+\sigma_e)$, there
exists $s_n\downarrow s$ such that $f(s_n)>f(s)$.
As $f(s)>f(t_e+\sigma_e)$, by continuity of $f$, we infer that there exists
$s^*\in(s,t_e+\sigma_e)$ such that
$f(s^*)=f(s)$. Therefore $\tau(s)\leq s^*\leq a$ as required.
\end{pf}

\begin{lemma}\label{rest}
We have
\[
\xi\bigl\{s\in C \dvtx\tau(s)\leq a\bigr\}=\eta(C)=0.
\]
\end{lemma}
\begin{pf}
First observe that if $s\in C$, then $f(t)\geq f(s)$ for all
$t\in(s,a]$. If $f(t)> f(s)$ for all $t\in(s,a]$, then $\tau(s)>a$.
Otherwise there exists a maximal
$t\in(s,a]$ with $f(s)=f(t)$. Then $f(t)$ is a true local minimum of
$f$ in the sense that there exists
$r>0$ with $f(s)\geq f(t)$ for all $s\in(t-r,t)$ and $f(s)> f(t)$
for all $s\in(t,(t+r)\wedge a)$. In particular
there are at most countably many levels $f(s)$ where this can happen.
Fixing such a level~$l$ we note that
$\xi\{s\in C \dvtx f(s)=l\}= \eta\{s\in C \dvtx f(s)=l\}$. Summing
over all these levels we see that
$\xi\{s\in C \dvtx\tau(s)\leq a\}\leq\eta(C)$. We conclude the
proof by showing that $\eta(C)=0$. Lemma~\ref{des}(b) ensures that,
for $\eta$-almost
every $s\in[0,a]$ there exists $s_n\downarrow s$ such that $f(s_n)<f(s)$,
which implies that $s\notin C$. Hence
the stated equality follows.
\end{pf}

\begin{pf*}{Proof of Lemma~\ref{analytic}}
Taking the sum over the equations in the previous two lemmas we obtain
$\xi\{s\geq0 \dvtx\tau(s)\leq a\}=\eta[0,a]$.
This implies
\[
\int\I\bigl\{0\leq\tau(s)\leq a\bigr\} \xi(ds)=\eta[0,a],\qquad a\geq0
\]
as any $s<0$ with $f(s)<0$ satisfies $\tau(s)\notin[0,a]$, and
Lemma~\ref{des}(a)
implies that $\xi$-almost every
$s<0$ with $f(s)\geq0$ satisfies $\tau(s)<0$, and so
$\xi\{s<0 \dvtx0<\tau(s)\leq a\}=0$.
\end{pf*}

\begin{pf*}{Proof of Theorem~\ref{main3}} Define
\begin{eqnarray*}
\xi^\infty&:=& \int\I\bigl\{\tau(s)=\infty, s \in\cdot\bigr\} \xi(ds),
\\
\eta^*&:=&\int\I\bigl\{\tau(s)\in\cdot\bigr\} \xi(ds).
\end{eqnarray*}
Recall that by Lemma~\ref{analytic} we have $\eta^*=\eta$ on $[s,\infty)$
provided $\xi[s,t]\ge\eta[s,t]$ for all $t\ge s$.
By Lemma~\ref{des} this holds
for $\xi^\infty$-a.e.~$s$. Moreover, by invariance of $\xi^\infty$ and
stationarity of $\BP$ we
have that $\BP\{\xi^\infty(-\infty,s]=0,\xi^\infty\ne0\}=0$
for all $s\in\R$. We infer that
%
%
\begin{equation}
\label{use1} \eta^*=\eta, \qquad\mbox{$\BP$-a.e. on $\bigl\{\xi^\infty\neq0
\bigr\}$.}
\end{equation}
Using the refined Campbell theorem \eqref{refC} twice, we obtain
%
%
\begin{eqnarray}
\label{campcalc} %
&&\BE_\BP \I\bigl\{
\xi^\infty\neq0\bigr\} \int_0^1 \I
\bigl\{\tau(s)<\infty\bigr\} \xi(ds)
\nonumber\\
&&\qquad = \BE_\BP\int_0^1 \I\bigl\{
\xi^\infty\circ\theta_s\neq0, \tau(\theta
_s,0)<\infty\bigr\} \xi(ds)
\nonumber\\
&&\qquad = \BQ_\xi \bigl\{ \xi^\infty\neq0, \tau(0)<\infty \bigr\}
\nonumber
\\[-8pt]
\\[-8pt]
\nonumber
&&\qquad = \BE_{\BQ_\xi} \int\I\bigl\{\xi^\infty\neq0, \tau(0)+s\in[0,1]
\bigr\} \,ds
\\
&&\qquad = \BE_\BP\int\I\bigl\{\xi^\infty\neq0, \tau(s)\in[0,1]
\bigr\} \xi(ds)\nonumber
\\
&&\qquad = \BE_\BP\I\bigl\{\xi^\infty\neq0\bigr\}
\eta^*[0,1].\nonumber
\end{eqnarray}
Using first \eqref{use1} and then our assumption gives
\[
\BE_\BP\I\bigl\{\xi^\infty\neq0\bigr\} \eta^*[0,1] =
\BE_\BP\I\bigl\{\xi^\infty\neq0\bigr\} \eta[0,1] =
\BE_\BP\I\bigl\{\xi^\infty\neq0\bigr\} \xi[0,1],
\]
and together with~\eqref{campcalc} we infer that
\[
\BE_\BP \I\bigl\{\xi^\infty\neq0\bigr\} \int
_0^1 \I\bigl\{\tau(s)<\infty\bigr\} \xi(ds) =
\BE_\BP\I\bigl\{\xi^\infty\neq0\bigr\} \xi[0,1],
\]
and therefore
$\tau(s)<\infty$ for $\xi$-a.e. $s$, $\BP$-a.e.
In particular, this implies that $\tau$ is a well-defined allocation rule.
An analogous argument implies that
\[
\tau^{-1}(s)>-\infty,\qquad \mbox{$\eta$-a.e. $s$, $\BP$-a.e.},
\]
where
\[
\tau^{-1}(s)=\sup\bigl\{t<s\dvtx\xi[t,s]=\eta[t,s]\bigr\}
\]
is the inverse of $\tau$.
We now use this to show that $\tau$ balances $\xi$ and $\eta$.
Fixing $a<b$ we aim to show that $\eta^*[a,b]=\eta[a,b]$.
If $f(t)\ge f(a)$ for all $t\in[a,b]$, this holds by
Lemma~\ref{analytic}.
Otherwise we apply this lemma to suitably
chosen alternative intervals. To this end let
\[
a^*:=\min\bigl\{s\in[a,b] \dvtx\mbox{$f(s)\leq f(t)$ for all $a\le t \le b$}
\bigr\}
\]
be the leftmost minimiser of $f$ on $[a,b]$.
As $\eta(a^*-\frac1n, a^*]\geq f(a^*-\frac1n)-f(a^*)
>0$ for all sufficiently large~$n\in\N$, we find a decreasing sequence $s_n$
with $\tau(s_n)\downarrow a^*$
and hence $f(s_n)\to f(a^*)$.
Then $s_n\downarrow s\in[-\infty,a]$ and $f(s)=f(a^*)$
if $s\neq-\infty$.

Assuming first that $s\neq-\infty$, we obtain from Lemma~\ref
{analytic} that
\[
\int\I\bigl\{\tau(s)\in\cdot\bigr\} \xi(ds)=\eta(\cdot),\qquad \mbox{on $[s,b]$},
\]
which implies the statement. Now assume that $s=-\infty$. In this case
we get $\eta^*=\eta$ on $[s_n,\tau(s_n)]$ and on $[a^*,b]$ for every $n$,
and the result follows as $n\to\infty$.
\end{pf*}

The following is a counterpart of Theorem~\ref{main3}
for simple point processes.

\begin{theorem}\label{tpp} Let $\xi$ and $\eta$ be invariant simple
point processes
on $\R$ defined on some probability space
equipped with a flow and an invariant $\sigma$-finite
measure. Assume that $\xi$ and $\eta$ have finite
intensities and that
\[
\BE_\BP\bigl[\xi[0,1]\mid\mathcal{I}\bigr]=\BE_\BP\bigl[
\eta[0,1]\mid\mathcal{I}\bigr].
\]
Then
%
%
\begin{equation}
\label{32} \tau(s):=\inf\bigl\{t\ge s \dvtx\xi[s,t]=\eta[s,t]\bigr\}, \qquad s\in\R,
\end{equation}
is an allocation rule balancing $\xi$ and $\eta$.
\end{theorem}

The 
allocation rule
$\tau$ in Theorem~\ref{tpp} is a one-sided
(and one-dimensional) version of the stable matching procedure
described in \cite{HPPS09}.
It can be proved by adapting the ideas of Theorem~\ref{main3}
to a discrete and therefore much simpler setup.

Theorem~\ref{main3} implies the following result which contains
Theorem~\ref{main2} as a special case.

\begin{theorem}\label{main2general}
Consider the Brownian case.
If $\mu$ and $\nu$ are orthogonal probability measures on $\R$,
then the stopping time
%
%
\begin{equation}
\label{bertoin3} T^{\mu, \nu}:=\inf \bigl\{t> 0\dvtx\ell^\mu[0,t]=
\ell^\nu[0,t] \bigr\}
\end{equation}
is an unbiased shift under $\BP_\mu$ and $\BP_\mu\{B_{T^{\mu, \nu}}\in
\cdot\}=\nu$.
\end{theorem}
\begin{pf}
Theorem~\ref{H1} implies that almost surely
$\BE_\BP[ \ell^\mu[0,1]| \mathcal I] =\break   \BE_\BP[ \ell^\nu[0,1]| \mathcal I]$.
By assumption and \eqref{0y} the invariant random
measures $\ell^\mu$ and $\ell^\nu$ are orthogonal.
Hence we can combine Theorems~\ref{main3} and \ref{main1a}
to obtain the result.
\end{pf}

%

\begin{remark}\label{rsub}
Assume in Theorem \ref{main2} that $\nu$ is a subprobability
measure. Then $T$ takes
the value $\infty$ with positive $\BP_0$-probability. Indeed,
by Remark~\ref{remsub}, defining the extended
allocation rule $\tau$ by $\tau(s):=s+T\circ\theta_s$
we get that $\tau$ balances
the restriction of $\ell^0$ to $\{s\dvtx\tau(s)<\infty\}$
and $\ell^\nu$. Assertion (i) of Theorem~\ref{main1a}
remains valid in the sense explained in Remark \ref{remsub}.
The embedding property $\BP_0\{T<\infty,B_T\in\cdot\}=\nu$
was proved in \cite{BertJan93}.
\end{remark}

\begin{remark}\label{rsuper}
Assume in Theorem~\ref{main2} that $\nu$ is a locally finite
measure with $\nu(\R)>1$ and $\nu\{0\}=0$. Then $\BP_0\{T<\infty\}=1$
and $\tau$ balances $\ell^0$ and $\eta:=\int\I\{\tau(s)\in\cdot\} \ell^0(ds)$.
The proof of Theorem~\ref{main3} still yields the inequality
$\eta\le\ell^\nu$. In particular $\eta$ is a diffuse (and invariant)
random measure with
intensity~$1$. The additive and continuous process
$(A_t)_{t\ge0}$ given by $A_t:=\eta[0,t]$ is adapted to the filtration
$(\sigma\{B_s\dvtx s\le t\})_{t\ge0}$.
However, since Theorem~22.25 in \cite{Kallenberg} applies only to
one-sided Brownian
motion, 
we cannot conclude that the process $(A_t)_{t\ge0}$ is of the form
$(\ell^{\nu'}[0,t])_{t\ge0}$ for some\vadjust{\goodbreak} probability measure $\nu'$, and
therefore it does not follow that the associated stopping time is an
unbiased shift.
The case $\nu=2\delta_1$ gives an example where it is easy to see that
this may
not be the case. Another example is discussed in Remark~\ref{rsuper2} below.
\end{remark}

\begin{remark}\label{rsuper2}
In \cite{BertJan93} stopping times of the form discussed in Remark~\ref{rsuper}
are used to embed a given probability measure $\nu'$ with $\int|x| \nu
'(dx)<\infty$
and $\nu'\{0\}=0$. Indeed, as in \cite{BertJan93}, page 547, define $\rho
(x):=2\int\I\{y>x\}(y-x) \nu'(dy)$ for $x\ge0$
and $\rho(x):=2\int\I\{y<x\}(x-y) \nu'(dy)$ for $x<0$. Let $m_0$ be the
maximum of the two numbers
$2\int^\infty_0 y \nu'(dy)$ and $-2\int^0_{-\infty} y \nu'(dy)$.
It is proved in \cite{BertJan93} that
%
%
\begin{equation}
\label{bertoin2} T:=\inf \biggl\{t> 0\dvtx\ell^0[0,t]<
m_0\int\ell^x[0,t]\rho^{-1}(x) \nu
'(dx) \biggr\}
\end{equation}
embeds $\nu'$ and satisfies $\BE_0 \ell^0[0,T]=m_0$.
This $T$ is
of the form \eqref{bertoin} with $\nu(\R)>1$, provided that $\rho>0$
$\nu'$-almost everywhere. This solution of the embedding problem is
optimal in the sense that
$\BE_0 \ell^x[0,S]\ge\BE_0 \ell^x[0,T]$, $x\in\R$, for any other
stopping time $S\ge0$
embedding $\nu'$. The idea of using first passage times of additive
functionals with
infinite Revuz measures to embed probability distributions goes back
to~\cite{monroe72b}.
The fact that $\BE_0 \ell^0[0,T]<\infty$ reveals that $T$ cannot be an
unbiased shift, as we show
in Theorem~\ref{t14lower} that this expectation is infinite for
unbiased shifts.
\end{remark}

The nonnegative unbiased shifts in Theorems~\ref{main2},
\ref{main2a} and in~\ref{littletwin}
are all stopping times. In the next example we construct a
nonnegative unbiased shift
embedding a distribution not concentrated
at zero, which is not a stopping time.

\begin{example}\label{exnonstop} Let $x\in\R\setminus\{0\}$.
We define an allocation rule $\tau$ that balances
$\ell^0$ and $\ell^x$ and such that $T:=\tau(0)$
is nonnegative but not a stopping time.
The mapping $\tau$ is the composition of
the following five allocation rules.
Let $\tau_1=\tau_4$ balance $\ell^0$ and $\ell^x$ according to
Theorem~\ref{main2}.
Let $\tau_2$ balance $\ell^x$ and $\ell^x$ by shifting
forward one mass-unit, that is, let $\tau_2(0)$ be defined
by \eqref{Tr} with $r=1$ and with $\ell^0$ replaced with $\ell^x$.
Let $\tau_3$ balance $\ell^x$ and $\ell^0$ according to
Theorem~\ref{main2}.
Finally define $\tau_5$ by shifting \emph{backwards} one mass-unit in
the local time at $x$; that is, let $\tau_5$ be defined
by \eqref{Tr} with $r=-1$ and $\ell^0$ replaced with $\ell^x$.
The composition $\tau$ of these allocation rules
balances $\ell^0$ and $\ell^x$. Moreover, $T:=\tau(0)\ge\tau_1(0)\ge0$.
However, $T$ is not a stopping time. This example
can be extended to a general target distribution~$\nu$.
\end{example}

\section{Target distributions with an atom at zero}\label{seczero}

In this section we prove Theorems~\ref{main2a}, \ref{prop3},
and \ref{p11}. In contrast to the previous section
we allow here for an atom at $0$.\vadjust{\goodbreak}

\begin{pf*}{Proof of Theorem~\ref{main2a}}
Let $y\in\R\setminus\{0\}$ such that
$\nu\{y\}=0$, and define
\[
\mu:=\nu-\nu\{0\}\delta_0+\nu\{0\}\delta_y.
\]
Theorems~\ref{main1} and \ref{main2} imply
that the allocation rule
\[
\tau'(s):=\inf \bigl\{t> s\dvtx\ell^0[s,t]=
\ell^\mu[s,t] \bigr\}, \qquad s\in\R,
\]
balances $\ell^0$ and $\ell^\mu$. The same theorems imply that there
is an allocation rule $\tau''$ that balances $\ell^y$ and $\ell^0$.
Define
\[
\tau(s):= \cases{ \tau'(s),&\quad $\mbox{if $B_{\tau'(s)}\ne y$,}$
\vspace*{2pt}
\cr
\tau''\bigl(\tau'(s)
\bigr), &\quad $\mbox{if $B_{\tau'(s)}= y$}.$}
\]
Then we have for any Borel set $C\subset\R$ outside
a fixed $\BP$-null set that
\begin{eqnarray*}
\hspace*{-4pt}&&\int\I\bigl\{\tau(s)\in C\bigr\} \ell^0(ds)
\\
\hspace*{-4pt}&&\qquad=\int\I\bigl\{\tau(s)\in C,B_{\tau'(s)}\ne y\bigr\} \ell^0(ds)
\\
\hspace*{-4pt}&&\qquad\quad{}+\int\I\bigl\{\tau(s)\in C,B_{\tau'(s)}= y\bigr\} \ell^0(ds)
\\
\hspace*{-4pt}&&\qquad=\int\I\bigl\{\tau'(s)\in C,B_{\tau'(s)}\ne y\bigr\}
\ell^0(ds) \\
\hspace*{-4pt}&&\qquad\quad{}+\int\I\bigl\{\tau''\bigl(
\tau'(s)\bigr)\in C,B_{\tau'(s)}= y\bigr\} \ell^0(ds)
\\
\hspace*{-4pt}&&\qquad=\int\I\{s\in C,B_s\ne y\} \ell^\mu(ds) +\int\I\bigl\{
\tau''(s)\in C,B_s= y\bigr\}
\ell^\mu(ds)
\\
\hspace*{-4pt}&&\qquad=\int\I\{x\ne0\}\I\{s\in C\} \ell^x(ds) \nu(dx) +\nu\{0\}\int\I
\bigl\{\tau''(s)\in C\bigr\} \ell^y(ds)
\\
\hspace*{-4pt}&&\qquad=\int\I\{x\ne0\}\I\{s\in C\} \ell^x(ds) \nu(dx) +\nu\{0\}
\ell^0(C)=\ell^\nu(C),
\end{eqnarray*}
where we have used \eqref{0y} (and $\nu\{y\}=0$) in
the penultimate equation.
Hence $\tau$ balances $\ell^0$ and $\ell^\nu$. Theorem~\ref{main1}
now implies that $T:=\tau(0)$ is an unbiased shift embedding $\nu$.
\end{pf*}

\begin{pf*}{Proof of Theorem~\ref{prop3}}
Let $T$ be any unbiased shift embedding $\nu$ and define $\tau:=\tau_T$.
Outside a fixed $\BP$-null set we obtain for any Borel set
$C\subset\R$ that
\begin{eqnarray*}
\int\I\bigl\{s\in C,\tau(s)=s\bigr\} \ell^0(ds) &=&\int\I\bigl\{
\tau(s)\in C,\tau(s)=s\bigr\} \ell^0(ds)
\\
&=&\int\I\bigl\{\tau(s)\in C,\tau(s)=s,B_{\tau(s)}=0\bigr\}
\ell^0(ds)
\\
&\le&\int\I\bigl\{\tau(s)\in C,B_{\tau(s)}=0\bigr\} \ell^0(ds)\\
& =&
\int\I\{s\in C,B_s=0\} \ell^\nu(ds)=\nu\{0\}
\ell^0(C),
\end{eqnarray*}
where we have used \eqref{0y} to obtain the final identity.
This implies that
\[
\I\bigl\{\tau(s)=s\bigr\}\le\nu\{0\}, \qquad\mbox{$\ell^0$-a.e. $s$, $
\BP$-a.e.}
\]
Assuming now that $\nu\{0\}<1$ we obtain
$\tau(s)\ne s$ for $\ell^0$-a.e. $s$, $\BP$-almost
everywhere. Lemma \ref{Palmlocal} now implies \eqref{21}.
\end{pf*}

\begin{pf*}{Proof of Theorem~\ref{p11}}
Let $\tau':=\tau_{T'}$,
where $T'$ is given by \eqref{bertoin} with
$\nu=p\delta_1+(1-p)\delta_2$. Define an invariant random measure
$\xi$ by $\xi(dt):=\I\{B_{\tau'(t)}=2\} \ell^0(dt)$.
The allocation rule
\[
\tau''(s):=\inf \bigl\{t> s\dvtx\xi[s,t]=1 \bigr\}
\]
balances $\xi$ with itself. Define
\[
\tau(s):=\cases{ s,&\quad $\mbox{if $B_{\tau'(s)}=1$,}$\vspace*{2pt}
\cr
\tau''(s),&\quad $\mbox{if $B_{\tau'(s)}= 2$}.$}
\]
It is easy to see that $\tau$ balances $\ell^0$ with itself.
Lemma \ref{Palmlocal} and Theorem~\ref{main1}
(or a direct calculation) implies that $T:=\tau(0)$ satisfies
\[
\BP_0\{T=0\}=\BP_0\{B_{\tau'(0)}=0\}=p.
\]
Since $T$ is an unbiased shift, the proof is complete.
\end{pf*}

\section{Stability and minimality of balancing allocations}\label{secstability}

We first work in the general setting of Section~\ref{secPalm}.
The following definition is a one-sided version of the notion of
stability introduced
in~\cite{HPPS09} for point processes. We call an allocation rule
$\tau\dvtx\Omega\times\R\rightarrow\R$ balancing $\xi$ and $\eta$
\emph{right-stable} if
$\tau(s)\geq s$ for all $s\in\R$ and
\[
\xi\otimes\xi \bigl\{ (s,t) \dvtx t<s\leq\tau(t) < \tau(s) \bigr\}=0,\qquad \BP
\mbox{-a.e.}
\]
Roughly speaking this means that the mass of pairs $(s,t)$ such that
$s$ would
prefer the partner of $t$ over its own partner, while $\tau(t)$ would
prefer $s$
over $t$ as a partner, vanishes.

\begin{theorem}\label{rstable}
Let $\xi$ and $\eta$ be invariant random measures satisfying the conditions
of Theorem~\ref{main3}, and suppose $\tau\dvtx\Omega\times\R\rightarrow
\R$
is the allocation rule constructed in the theorem. Then $\tau$ is right-stable.
\end{theorem}
\begin{pf}
By Lemma~\ref{des}(a) and continuity of $f$, we have for $\xi$-a-e.
$s$ that
$f(s)<f(r)$ for all $r\in(s,\tau(s))$. Hence $\xi\otimes\xi$-almost
every pair
$(s,t)$ with $t<s\leq\tau(t) < \tau(s)$ satisfies $f(t)<f(s)<f(\tau(t))$
contradicting the definition of $\tau$.
\end{pf}

Right-stable allocation rules have a useful minimality property.

\begin{theorem}\label{mini}
Any right-stable allocation rule $\tau$ balancing two measures $\xi$
and $\eta$
is minimal in the sense that if $\sigma$ is another allocation rule
balancing $\xi$ and $\eta$
such that $s\leq\sigma(s) \leq\tau(s)$ for $\xi$-almost every $s\in\R
$, then
$\xi\{ s \dvtx\sigma(s)<\tau(s)\}=0$.
\end{theorem}
\begin{pf}
By right-stability of $\tau$ we have, \textsl{for $\xi$-almost every $a$,}
%
%
\begin{equation}
\label{52} s\in\bigl[a,\tau(a)\bigr] \quad\Longleftrightarrow\quad \tau(s)\in\bigl[a,\tau(a)
\bigr],\qquad \xi\mbox{-a.e. $s$}.
\end{equation}
From the assumption $s\leq\sigma(s) \leq\tau(s)$ and \eqref{52} we obtain
for any $t\in[a,\tau(a)]$ that
$\tau(s)\in[a,t]$ implies $\sigma(s)\in[a,t]$ for $\xi$-almost every
$s$. Therefore
\[
\eta[a,t]=\int\I\bigl\{\tau(s)\in[a,t]\bigr\} \xi(ds) \le\int\I\bigl\{\sigma(s)
\in[a,t]\bigr\} \xi(ds) =\eta[a,t].
\]
This implies
\[
\I\bigl\{\tau(s)\in[a,t]\bigr\}= \I\bigl\{\sigma(s)\in[a,t]\bigr\},\qquad\xi
\mbox{-a.e. $s\in\R$}.
\]
Therefore $\tau$ and $\sigma$ coincide $\xi$-almost everywhere
on $\tau^{-1}([a,\tau(a)])$.

Now fix some $b\in\R$ and recall the definition of the backwards running
minimum $m(t)= \min\{ f(s) \dvtx t \leq s \leq b\}$ and the set
$C=\{t\leq b \dvtx m(t)=f(t)\}$.
We have seen that the complement of $C$ consists of countably
many intervals $(a,\tau(a))$
as above and therefore $\tau$ and $\sigma$ coincide $\xi$-almost
everywhere on $\tau^{-1}((-\infty,b]\setminus C)$.
On the other hand, by Lemma~\ref{rest} we have
$\xi(\tau^{-1}(C))=\eta(C)=0$, as required to finish the argument.
\end{pf}

\begin{remark}
In the point process case the allocation rule \eqref{32} is right-stable,
and it is not difficult to show
that it is the unique right-stable allocation balancing~$\xi$ and $\eta$.
We conjecture that this uniqueness property also holds in the general case
and therefore can be added to Theorem~\ref{rstable}.
\end{remark}

\begin{remark}
One could define an allocation rule $\tau$ to be \emph{stable} if
\[
\xi\otimes\xi \bigl\{ (s,t) \dvtx \bigl|s-\tau(t)\bigr| < \bigl|s-\tau(s)\bigr|, \bigl|s-\tau(t)\bigr| <\bigl|t-
\tau(t)\bigr| \bigr\}=0.
\]
The rule $\tau$ of Theorem~\ref{main3} does not satisfy this. We
do not know if stable allocation rules in the above sense exist, or
if they are unique.
\end{remark}

In the remainder of this section we consider the Brownian case.
An unbiased shift $T$ is called \emph{minimal unbiased shift}
if $\BP_0\{T\ge0\}=1$ and if
any other unbiased shift $S$ such that
$\BP_0\{0\le S\le T\}=1$\vadjust{\goodbreak} and $\BP_0\{B_T\in\cdot\} = \BP_0\{B_S\in
\cdot\}$
satisfies $\BP_0\{S=T\}=1$.
The following theorem provides more insight
into the set of all minimal unbiased shifts.
The result and its proof are motivated by Proposition 2 in
\cite{monroe72}.

\begin{theorem}\label{pminexist} Let $T$ be an unbiased
shift embedding the probability measure $\nu$
and such that $\BP_0\{T\ge0\}=1$. Then there exists
a minimal unbiased shift $T^*$ embedding $\nu$ and such that
$\BP_0\{0\leq T^*\le T\}=1$.
\end{theorem}
\begin{pf} Let $\mathcal{T}$ denote the set of all
unbiased shifts $S$ embedding $\nu$ and
such that $\BP_0\{0\le S\le T\}=1$.
This is a partially ordered set, where we do not
distinguish between elements that coincide $\BP_0$-a.s.
By the Hausdorff maximal principle
(see, e.g.,~\cite{Dudley}, Section~1.5) there is a \emph{maximal chain}
$\mathcal{T}'\subset\mathcal{T}$. This is a totally ordered set
that is not contained in a strictly bigger totally ordered set.
Let
\[
\alpha:=\sup_{S\in\mathcal{T}'} \BE_0 e^{-S}.
\]
Then there is a sequence $S_n$, $n\in\N$, such that
$\BE_0 e^{-S_n}\to\alpha$ as $n\to\infty$. Since $\mathcal{T}'$
is totally ordered, it is no restriction of generality
to assume that the $S_n$ are decreasing $\BP_0$-a.s.
Define $T^*:=\lim_{n\to\infty}S_n$.
By construction and monotone convergence,
%
%
\begin{equation}
\label{6.3} \BE_0 e^{-T^*}=\alpha.
\end{equation}
We also note that $\BP_0\{0\le T^*\le T\}=1$.

We claim that $T^*$ is a minimal unbiased shift embedding $\nu$
and first show that $T^*$ is an unbiased shift.
Let $k\in\N$, and consider continuous and bounded functions
$f\dvtx\R^k\rightarrow\R$ and $g\dvtx\R\rightarrow\R$.
Let $t_1,\dots,t_k\in\R$. Since $S_n\in\mathcal{T}$
for any $n\in\N$, we have that
%
%
\begin{eqnarray}
&&\BE_0 f(B_{S_n+t_1}-B_{S_n},\dots,B_{S_n+t_k}-B_{S_n})g(B_{S_n})
\nonumber
\\[-8pt]
\\[-8pt]
\nonumber
&&\qquad=\BE_0 f(B_{t_1},\dots,B_{t_k})\int g(x) \nu(dx).
\end{eqnarray}
By bounded convergence the above left-hand side
converges toward
\[
\BE_0 f(B_{T^*+t_1}-B_{T^*},\dots,B_{T^*+t_k}-B_{T^*})g(B_{T^*})
\]
as $n\to\infty$. The monotone class theorem implies
that $T^*$ is an unbiased shift embedding~$\nu$.

It remains to show the minimality property of $T^*$.
Assume on the contrary that there is some
unbiased shift $S$ embedding $\nu$ such
that $\BP_0\{0\le S\le T^*\}=1$
and $\BP_0\{S<T^*\}>0$. The last two relations
imply that
\[
\BE_0 e^{-S}>\BE_0 e^{-T^*}.
\]
By \eqref{6.3} this means that $S\notin\mathcal{T}'$.
On the other hand, since $\BP_0\{S\le T^*\le T\}=1$,
we have that $S\in\mathcal{T}$, contradicting
the maximality property of $\mathcal{T}'$.
\end{pf}

As announced in the \hyperref[sintro]{Introduction} the stopping time
$T^\nu$
is a minimal unbiased shift:

\begin{theorem}\label{pminimal}
Let $\nu$ be a probability measure on $\R$ with $\nu\{0\}=0$.
Then
$T^\nu$ defined by \eqref{bertoin} is a minimal
unbiased shift.
\end{theorem}
\begin{pf}
Let $S$ be an unbiased shift embedding $\nu$ and such that
$\BP_0\{0\le S\le T^\nu\}~=~1$.
Theorem~\ref{main1}
implies that the allocation rules $\tau_S$ and $\tau_{T^\nu}$ balance
$\ell^0$ and $\ell^\nu$. By Theorem~\ref{rstable}, $\tau_{T^\nu}$ is
right-stable
$\BP$-a.e. The assumptions yield
$\ell^0\{s\dvtx s\le\tau_S(s)\le\tau_{T^\nu}(s)\}=0$ $\BP$-a.e.
By Theorem~\ref{mini}
we therefore have $\ell^0\{s \dvtx\tau_S(s)<\tau_{T^\nu}(s)\}=0$ $\BP$-a.e.
This readily implies that $\BP_0\{S=T^\nu\}=1$.
\end{pf}

\section{Moments of unbiased shifts}\label{secproperties}

In this section we consider the Brownian case and discuss moment
properties of unbiased shifts. The following
two theorems together were stated as Theorem~\ref{t14complete} in the
\hyperref[sintro]{Introduction}.

\begin{theorem}\label{t14lower}
Suppose $\nu$ is a target distribution with $\nu\{0\}=0$, and the
stopping time $T\geq0$ is an
unbiased shift embedding~$\nu$. Then
\[
\BE_0 T^{1/4}=\infty.
\]
\end{theorem}
\begin{pf}
We start the proof with a reminder of the Barlow--Yor inequality \cite{BY},
which states that, for
any $p>0$ there exist constants $0<c<C$ such that, for all stopping times~$T$,
\[
c \BE_0T^{p/2} \le\BE_0 \sup
_x \ell^x[0,T]^{p} \le C
\BE_0 T^{p/2}.
\]
Hence it suffices to verify that $\BE_0\ell^0[0,T]^{1/2}=\infty$.

The proof of this fact uses an argument similar to that in the proof
of Theorem~2 in~\cite{HPPS09}. Let $\tau=\tau_T$ be the allocation rule
associated
with $T$, and
set $T_r=\sup\{s \geq0 \dvtx\ell^0[0,s]=r\}$ for $r > 0$,
as at \eqref{Tr}. Then, on the one hand,
\begin{eqnarray*}
&&\BE_0\int\I\bigl\{0\le s\le T_r,\tau(s)
\notin[0,T_r]\bigr\} \ell^0(ds)\\
&&\qquad = \BE_0\int
^{T_r}_0\I\bigl\{\tau(s)-s>T_r-s\bigr\}
\ell^0(ds)
\\
&&\qquad=\int^r_0\BP_0\bigl\{
\tau(T_s)-T_s>T_r-T_s\bigr\} \,ds
=\int^r_0\BP_0\{T\circ
\theta_{T_s}>T_{r-s}\circ\theta_{T_s}\} \,ds
\\
&&\qquad=\int^r_0\BP_0
\{T>T_s\} \,ds=\BE_0\bigl[\ell^0[0,T]\wedge r
\bigr],
\end{eqnarray*}
where we have used the strong Markov property at $T_s$
(or Theorem~\ref{littletwin}) 
for the fourth step and change of variable for the second and fifth steps.
On the other hand, the fact that $\tau$ balances $\ell^0$ and $\ell^\nu$
easily implies that
\[
\int\I\bigl\{0\le s\le T_r,\tau(s)\notin[0,T_r]\bigr\}
\ell^0(ds)\ge \bigl(\ell^0[0,T_r]-
\ell^\nu[0,T_r]\bigr)_+.
\]
Hence, combining these two facts with the obvious fact that $\ell
^0[0,T_r]=r$, we get
%
%
\begin{equation}
\label{CL1} \BE_0\bigl[\ell^0[0,T]\wedge r\bigr]\ge
\BE_0 \bigl(r-\ell^\nu[0,T_r]\bigr)_+.
\end{equation}
We now show that
%
%
\begin{equation}
\label{CL2} \liminf_{r\to\infty}r^{-1/2}\BE_0
\bigl(r-\ell^\nu[0,T_r]\bigr)_+>0.
\end{equation}
To this end we apply a concentration inequality of Petrov for arbitrary
sums of
independent random variables; see \cite{Petrov95}, Theorem~2.22.
It shows that there exists a constant $C>0$ such that, for all $\eps>0$
and $r\geq1$,
\[
\BP \bigl\{ \ell^\nu[0,T_r] \in[r-\eps\sqrt{r}, r+\eps
\sqrt{r}] \bigr\} \leq C \eps.
\]
Now observe that $\BE_0\ell^\nu[0,T_r]=r$, which is an immediate consequence
of the second Ray--Knight theorem (see Theorem~2.3 in Chapter XI of
\cite{RevYor99}) but can also be derived from general Palm theory.
Hence, by
Markov's inequality,
\begin{eqnarray*}
\BE_0 \bigl(r-\ell^\nu[0,T_r]\bigr)_+ & =&
\tfrac12 \BE_0 \bigl|r-\ell^\nu[0,T_r]\bigr| \geq
\tfrac12 \eps\sqrt{r} \BP \bigl\{ \bigl|r-\ell^\nu[0,T_r]\bigr|
> \eps \sqrt{r} \bigr\} \\
&\geq&\tfrac12 \eps(1-C\eps) \sqrt{r}
\end{eqnarray*}
as required to prove~\eqref{CL2}. Combining~\eqref{CL1} and~\eqref{CL2} gives
\[
\liminf_{r\to\infty}r^{-1/2} \BE_0\bigl[
\ell^0[0,T]\wedge r\bigr]>0.
\]
Finally, assume for contradiction that $\BE_0[\ell^0[0,T]^{1/2}]<\infty$.
Since\break  $r^{-1/2} (\ell^0[0,T]\wedge r)\le\ell^0[0,T]^{1/2}$, dominated
convergence
implies that\break  $r^{-1/2}\BE_0[\ell^0[0,T]\wedge r]\to0$ as $r\to\infty$, which
is in contradiction to the last display.
\end{pf}

Note that the unbiased shifts
$T^\nu$ satisfy the conditions of Theorem~\ref{t14lower}
if $\nu$ has finite mean. The next result shows that they have
nearly optimal moment properties.

\begin{theorem}\label{t14}
Let $\nu$ satisfy $\int|x| \nu(dx)<\infty$,
and let $T = T^\nu$ be the stopping time constructed in~\eqref{bertoin}.
Then, for all $\beta\in[0,1/4)$,
%
%
\begin{equation}
\label{tail} \BE_0 T^{\beta}<\infty.
\end{equation}
\end{theorem}

The proof of Theorem~\ref{t14} uses a result similar
to Theorem 4(ii) in~\cite{HP05} and Theorem~2 in~\cite{HPPS09},
which is of independent interest and may also\vadjust{\goodbreak} serve as another
example for Theorem~\ref{main3}. We consider the ``clock''
\[
U_r:=\inf \bigl\{ t>0 \dvtx\ell^0[0,t]+
\ell^\nu[0,t]=r \bigr\}
\]
and random measures $\xi$ and $\eta$ on the positive reals given by
\[
\xi[0,r]:=\ell^0[0,U_r],\qquad \eta[0,r]:=
\ell^\nu[0,U_r], \qquad r\geq0.
\]

\begin{proposition}\label{t41} Let $\xi$ and $\eta$ be defined as above,
and let $S:=\inf\{t>0 \dvtx\xi[0,t]=\eta[0,t]\}$. Then $\BE_0
S^{1/2}=\infty$,
but for some $c>0$ we have $\BP_0\{S>t\}\le ct^{-1/2}$, for all $t\in\R$.
\end{proposition}

\begin{pf}
The proof of
$\BE_0 S^{1/2}=\infty$ is very similar to Theorem~2
in \cite{HPPS09} and is therefore omitted. We prove here the upper
bound for the
tail asymptotics (only this part is needed).
This result is similar to Theorem~6(ii) in \cite{HPPS09}, but
due to the specific form of $S$ we can use a more direct argument.

For any $i\in\N$ let $Y_i=\eta\{ s\geq0 \dvtx i-1 < \xi[0,s]\leq i\}$.
As in the proof of Theorem~\ref{t14lower}
the second Ray--Knight theorem implies that $Y_i$ has mean one.
Together with Jensen's inequality we get
%
%
\begin{equation}
\label{vari}\qquad \BE_0 \bigl[Y_i^2\bigr] =
\BE_0 \bigl(\ell^\nu[0,T_1]\bigr)^2
\le\BE_0 \int\bigl(\ell ^x[0,T_1]
\bigr)^2 \nu(dx) =\int\bigl(1+|x|\bigr) \nu(dx),
\end{equation}
which is finite by assumption. Summarising, the sequence
$Y_1, Y_2, \ldots$ is an i.i.d. sequence of random variables with mean
one and
finite variance.
Define, for~$n\in\N$,
\[
R_n:=\sum_{i=1}^n
1+Y_i,\qquad U_n:=\sum_{i=1}^n
1-Y_i.
\]
Let $\sigma:=\inf\{n\ge1 \dvtx U_n<0\}$, and
fix $a\in(0,1/2)$. Then, for any $t>0$,
\[
\BP_0\{S>t\}\le\BP_0\{R_\sigma>t\}\le
\BP_0\{\mbox{$U_n \ge0$ for all $n\le at$}\} +
\BP_0\{R_{\lfloor at\rfloor}>t\}.
\]
By a classical result of Spitzer~\cite{Spitzer60}, see also
\cite{Feller}, Theorem~1a in Section~XII.7,
the first term on the above right-hand side
is bounded by a constant multiple of $(at)^{-1/2}$.
By Chebyshev's inequality, we have
\[
\BP_0\{R_{\lfloor at\rfloor}>t\} \leq\frac1{\bigl(t-2\lfloor at\rfloor
\bigr)^2} \BE_0 \bigl[\bigl(R_{\lfloor at\rfloor}-2\lfloor at
\rfloor\bigr)^2 \bigr] = \frac{\lfloor at\rfloor}{(t-2\lfloor at\rfloor)^2} \BE_0
\bigl[(1-Y_1)^2\bigr],
\]
which is bounded by a constant multiple of $t^{-1}$. This completes the proof.
\end{pf}

\begin{pf*}{Proof of Theorem~\ref{t14}} The variable $S$, defined
in Proposition~\ref{t41}, satisfies
\[
S=\ell^0[0,T]+\ell^\nu[0,T]=2\ell^0[0,T].
\]
It remains to relate the tail behaviour of
$\ell^0[0,T]$ (which we know) to that of
$T$ (which we require).
To this end we observe\vadjust{\goodbreak} that for $\theta\in\R$ and $t>\frac34$,
using \cite{MortersPeres}, Theorem~6.10,
\begin{eqnarray*}
&&\BP_0 \biggl\{\inf_{s>t} \frac1{\sqrt{s/\log s}}
\ell^0[0,s] < 1/\theta \biggr\}\\
&&\qquad = \BP_0 \biggl\{\inf
_{s>t} \frac1{\sqrt{s/\log s}} \max_{0\leq r\leq s}
|B_r| < 1/\theta \biggr\}
\\
&&\qquad\leq\sum_{k=0}^\infty\BP_0
\biggl\{ \frac1{\sqrt{t+k}} \max_{0\leq r\leq t+k} |B_r| <
\frac{2}{\theta\sqrt{\log(t+k)}} \biggr\}.
\end{eqnarray*}
By a step in the proof of Chung's law of the iterated logarithm
(see, e.g.,~\cite{JP75}, equation~(2.1)),
\[
\BP_0 \biggl\{\frac1{\sqrt{t}}\max_{0\leq r\leq t}
|B_r| <x \biggr\} \le\frac{4}{\pi} e^{- {\pi^2}/{(8x^2)}},\qquad x>0,
\]
and hence we have
\[
\BP_0 \biggl\{\inf_{s>t} \frac1{\sqrt{s/\log s}}
\ell^0[0,s] < 1/\theta \biggr\} \le t^{-1/4}
\]
for a sufficiently large constant $\theta$. For sufficiently large~$t$
we have
\[
\BP_0 \biggl\{ \frac{T}{\theta^2 \log T}>t \biggr\} \le\BP_0
\bigl\{\ell^0[0,T]>\sqrt{t}\bigr\} +\BP_0 \biggl\{\inf
_{s>t} \frac1{\sqrt{s/\log s}}\ell^0[0,s] < 1/
\theta \biggr\},
\]
and the right-hand side in this inequality
is bounded by a constant multiple of $t^{-1/4}$.
The result follows directly by integration.
\end{pf*}

Next we turn to unbiased shifts $T$ embedding
a measure $\nu\neq\delta_0$, which need neither be stopping times,
nor nonnegative. We conjecture that any such
shift satisfies $\BE_0 |T|^{1/4}=\infty$. At the moment we
can only prove the following weaker result.

\begin{theorem}\label{pmoment}
If $T$ is an unbiased shift embedding a probability measure
$\nu\neq\delta_0$, then
\[
\BE_0\sqrt{|T|}=\infty.
\]
\end{theorem}
\begin{pf}
The idea of this proof is due to Alex Cox. We work under the
probability measure $\BP_0$.
By definition of an unbiased shift
$B^+:=(B_{T+ t}-B_T \dvtx t\ge0)$ and
$B^-:=(B_{T-t}-B_T \dvtx t\ge0)$ are independent
Brownian motions. Moreover, the pair $(B^+,B^-)$
is independent of $B_T$. Assume that
$B_T\ge x$, where $x>0$ is chosen such that
$\nu[x,\infty)>0$. (If there is no such $x>0$ we find
an $x<0$ such that $\nu(-\infty,x]>0$ and
assume $B_T\le x$.) If $T>0$, then $B^-_T=-B_T\le-x$,
so that
\[
T\ge S^-:=\inf\bigl\{t\ge0\dvtx B^-_t=-x\bigr\}.
\]
If $T<0$, then $B^+_{-T}=-B^-_T\le-x$,
so that
\[
-T\ge S^+:=\inf\bigl\{t\ge0\dvtx B^+_t=-x\bigr\}.
\]
Hence $|T|\ge S^-\wedge S^+=:S$. It is well known
that $\BE_0\sqrt{S^-}=\infty$ and $\BE_0\sqrt{S^+}=\infty$.
Since $S^-$ and $S^+$ are independent, this property
transfers to $S$. It follows that
%
%
\begin{equation}
\BE_0\sqrt{|T|}\ge\BE_0\I\{B_T
\ge x\}\sqrt{S} =\nu[x,\infty) \BE_0\sqrt{S}=\infty.
\end{equation}
\upqed\end{pf}

Unbiased shifts embedding $\delta_0$ also have bad moment properties
if they are nonnegative (or, by time-reversal, nonpositive) but not
identically zero. The result can be compared with
Theorem~3(i) in \cite{HPPS09}. However, the proofs are very different.

\begin{theorem}\label{posmoment}
If $T\ge0$ is an unbiased shift such that $\BP_0\{B_T=0\}=1$ and
$\BP_0\{T>0\}>0$, then
\[
\BE_0T=\infty.
\]
\end{theorem}
\begin{pf}
We assume for contradiction that $m:=\BE_0 T<\infty$.
Define a probability measure $\BP^*$ on $\Omega$ by
setting
$
\BE_{\BP^*} f(B)= \frac1m \BE_0\int_0^{T} f(\theta_s B) \,ds
$
for each bounded nonnegative measurable function~$f$.
By Lemma~\ref{secondQ*}, $\BP^*$ is stationary.
%
%
To show
that, on the invariant $\sigma$-algebra~$\mathcal I$,
the process $B$ has
the same distribution under $\BP^*$ as under $\BP_0$
, take
$A\in\mathcal I$
and recall from Theorem~\ref{H1} that
$\BP_0\{B\in A\}\in\{0,1\}$.
But $\BP^*\{B\in A\}=\frac1m \BE_0 \I\{B\in A\}T=0$ or $1$
according as $\BP_0\{B\in A\}=0$ or $1$, as required. By
\cite{Thor96}, Theorem~2, we infer from this that
\[
\frac1t \int_0^t \BP_0\{
\theta_sB\in\cdot\} \,ds \to \BP^* \{B\in\cdot\},\qquad t\to\infty,
\]
with respect to the total variation norm. On the other hand, for every $r>0$,
\[
\frac1t \int_0^t \BP_0
\bigl\{|B_s|\leq r\bigr\} \,ds \to0,\qquad  t\to\infty,
\]
implying $\BP^*\{|B_0|\le r \}=0$ for all $r>0$, which is a
contradiction.
\end{pf}

In contrast to the two
theorems above,
we shall see below that unbiased shifts can have \emph{good} moment
properties if they can assume both signs.

\begin{example}\label{exexpon}
We construct a nonzero unbiased shift $T$ embedding $\delta_0$, which has
$\BE e^{\lambda|T|}<\infty$ for some $\lambda>0$.
Let $\{(a_i,b_i) \dvtx i\in\Z\}$ be the countable
collection of maximal nonempty intervals $(a,b)$ with
the property that $B_t\neq0$ for all $a<t<b$ and $|B_s|\geq1$ for some
$s\in(a,b)$.
We assume that the collection is ordered
such that $b_i<a_{i+1}$ for all~$i\in\Z$. We define an allocation rule
$\tau$
by the requirement that, for $b_i<s<a_{i+1}$,
\[
\tau(s)= \cases{ \sup\bigl\{r<a_{i+1}\dvtx\ell^0(r,a_{i+1})=
\ell^0(b_i,s)\bigr\}, &\quad $\mbox{if $\ell^0(b_i,s)
\leq\frac12\ell^0(b_i,a_{i+1})$},$\vspace*{2pt}
\cr
\inf\bigl\{r>b_{i}\dvtx\ell^0(s,a_{i+1})=
\ell^0(b_i,r)\bigr\}, &\quad  $\mbox{if $\ell^0(b_i,s)>
\frac12\ell^0(b_i,a_{i+1})$}.$}
\]
It is easy to see that $\tau$ balances $\ell^0$ with itself,
and hence by Theorem~\ref{main1}, we have that
$T=\tau(0)$ is an unbiased shift embedding $\delta_0$. Moreover,
we have $|T|\le S_1+S_2$ where
$S_1=\inf\{t>0 \dvtx |B_t|=1\}$ and $S_2=-\sup\{t<0 \dvtx\break |B_t|=1\}$.
$S_1$ and $S_2$ are obviously
independent and identically distributed, and it is easy to see that they,
and hence $|T|$, have the required moment property.
\end{example}

\begin{remark} If $T\geq0$ is an unbiased shift such that $\BP_0\{
B_T=0\}=1$
and $\BP_0\{T>0\}>0$, then
we conjecture that $\BE_0\sqrt{T}=\infty$ (strengthening Theorem~\ref
{posmoment}),
but we cannot prove this
without additional assumptions. One such
assumption [covering $T_r$ defined in \eqref{Tr} for $r>0$]
is that $\BP_0\{T>s\}>0$ for some $s>0$
such that $\{T>s\}$ is $\BP_0$-almost surely in the $\sigma$-algebra
generated by $\{B_t\dvtx t\le s\}$.
Indeed, in this case we have
\[
\BE_0 \sqrt{|T|}\ge\BE_0 \I\{T >s\}\sqrt{T} \ge
\BE_0 \I\{T >s\}\sqrt{s+T_0\circ\theta_s},
\]
where $T_0:=\inf\{t>0\dvtx B_t=0\}$. By the Markov property
\[
\BE_0 \sqrt{|T|}\ge\BE_0 \I\{T >s\}\BE_{B_s}
\sqrt{T_0}=\infty,
\]
since $\BE_{x}\sqrt{T_0}=\infty$ for all $x\ne0$ and
$\BP_0\{B_s=0\}=0$. Note that this argument does not use that $T$ is unbiased.
\end{remark}

\section{Unbiased shifts of L\'evy processes}\label{secremarks}

In this section we extend some of our previous results
to a larger class of L\'evy processes.
A \emph{L\'evy process} is a right-continuous real-valued stochastic process
$X=(X_t)_{t\in\R}$ with left-hand limits and $X_0=0$,
having independent and stationary increments; see, for example, \cite
{Bert02,Kallenberg}.
In particular the (left-continuous) process $(X_{-t})_{t\ge0}$ is
independent of $X^+:=(X_t)_{t\ge0}$
and has the same finite-dimensional distributions as $-X^+$. We assume
that $X$ is \emph{recurrent};
see \cite{Bert02} for a definition.

For convenience, we also assume that $X$ is given as the identity on
its canonical space
$(\Omega,\mathcal{A},\BP_0)$, where $\Omega$ is the set
of all right-continuous functions $\omega\dvtx\R\rightarrow\R$ with left-hand
limits and $\mathcal{A}$ is the Kolmogorov product $\sigma$-algebra.
As in the Brownian case we define $\BP_x:=\BP_0\{X+x\in\cdot\}$, $x\in\R$,
and $\BP$ by \eqref{statP}. This $\BP$ has the stationarity
property \eqref{Pinv}, where the shifts are defined by \eqref{shiftB}.
This setting is a special case of the one established in
Section~\ref{secPalm}.

The L\'evy--Khinchine formula
states that
%
%
\begin{equation}
\label{LC} \BE_0 e^{i\vartheta X_t}=e^{-t\psi(\vartheta)},\qquad  t\ge0,
\end{equation}
where
\[
\psi(\vartheta)=ia\vartheta+\frac{\sigma^2\vartheta^2}{2} +\int \bigl(1-e^{i\vartheta x}+i
\vartheta x\I\bigl\{|x|\le1\bigr\} \bigr) \Pi(dx), \qquad\vartheta\in\R.
\]
Here $a\in\R$, $\sigma^2\ge0$ and the \emph{L\'evy measure}
$\Pi$ satisfies $\Pi\{0\}=0$ and $\int x^2\wedge1 \Pi(dx)<\infty$.
We assume that, first,
\[
\int\Re \biggl(\frac{1}{u+\psi(\vartheta)} \biggr) \,d\vartheta<\infty,\qquad u>0,
\]
which means that points are not essentially polar,
and, second, that either $\sigma^2>0$ or $\int|x|\wedge1 \Pi
(dx)=\infty$,
which means that the L\'evy process is of unbounded variation.
These two assumptions imply that the origin is regular for itself;
see Theorem~8 in~\cite{Bret71}. Theorem~4 in \cite{GeKe72} then
implies that there are random (local time) measures $\ell^x$, $x\in\R$,
such that $(\omega,x)\mapsto\ell^x(\omega,C)$ is measurable
for all Borel sets $C\subset\R$, and \eqref{occupation} holds.
Moreover, $\ell^x$ is $\BP_0$-a.s. diffuse for any $x\in\R$.
In order to apply the techniques of this paper we need a \emph{perfect}
version of local times satisfying \eqref{0inv}, \eqref{0x},
and \eqref{0y}. To achieve this we assume the conditions
$(R_\beta)$ and $(H)$ of \cite{BPT86}, Theorem 1.2.
We do not formulate these (somewhat technical) assumptions here, but
only mention that
they are satisfied by a strictly $\alpha$-stable L\'evy process,
whenever $\alpha>1$.
[The case $\alpha=2$ corresponds to Brownian motion
while for $\alpha<2$ the L\'evy measure is given by
$\Pi(dx)=c_+x^{-\alpha-1} \,dx$ on $(0,\infty)$
and $\Pi(dx)=c_-|x|^{-\alpha-1} \,dx$ on $(-\infty,0)$.]

As in the Brownian case we define
for any locally finite measure $\mu$ on $\R$
the invariant random measure $\ell^\mu$ by \eqref{ellnu}.
If $\mu$ is a probability measure, then we call
a random time $T$
an \emph{unbiased shift} under $\BP_\mu:=\int\BP_x \mu(dx)$ if
$(X_{T+t}-X_T)_{t\in\R}$ is independent of~$X_T$
and has distribution $\BP_0$ under $\BP_\mu$.

\begin{theorem}\label{main1b} Let $T$ be a random time
and $\mu,\nu$ be probability measures on~$\R$. Then
$T$ is an unbiased shift under $\BP_\mu$ and
$\BP_\mu\{X_T\in\cdot\}=\nu$ if and only if
the allocation rule $\tau_T$ defined
by \eqref{tauT} balances $\ell^\mu$ and $\ell^\nu$.
\end{theorem}
\begin{pf}
The proof of Lemma \ref{Palmlocal} yields
that $\BP_x$ is the Palm measure of $\ell^x$
with respect to $\BP$.
Therefore $\BP_\mu$ is the Palm measure of $\ell^\mu$,
and the proof of Theorem~\ref{main1a} applies without
change.
\end{pf}

\begin{theorem}\label{twin1levy}
Let $\mu$ be a probability measure
on $\R$, and let $S_r$, $r \in\R$, be the generalised~inverse of
$\ell^\mu$ defined as in~\eqref{Tr}. Then
$S_r$ is an unbiased shift under $\BP_\mu$ and
$\BP_\mu\{X_{S_r}\in\cdot\}=\mu$.
\end{theorem}
\begin{pf}
In order to apply Theorem \ref{mass-stat} we need
to show that
$\BP_\mu\{\ell^\mu(0,\break \infty)<\infty\}=0$.
Since the L\'evy process $-X$
also satisfies
our general assumptions, this implies
$\BP_\mu\{\ell^\mu(-\infty,0)<\infty\}=0$.
Clearly, it is enough to prove $\BP_x\{\ell^y(0,\infty)<\infty\}=0$
for all $x,y\in\R$. By the spatial homogeneity \eqref{0x}
this is equivalent to
%
%
\begin{equation}
\label{9.7} \BP_0\bigl\{\ell^x(0,\infty)<\infty\bigr\}=0,\qquad
x\in\R.
\end{equation}
By Proposition V.4 in \cite{Bert02} the generalised
inverse of $(\ell^0[0,t])_{t\ge0}$ is (under $\BP_0$) a (finite)
subordinator, so that \eqref{9.7} holds for $x=0$.
Spatial homogeneity implies that
$\BP_x\{\ell^x(0,\infty)<\infty\}=0$.
As the origin is regular for itself, the results in \cite{Bert02}, Chapter II
(see in particular Theorems~II.16 and II.19)
imply that our process is not only recurrent, but
that the origin is \emph{point-recurrent} and that
$T'_x:=\inf\{t\ge0\dvtx X_t=x\}$ is finite $\BP_0$-a.s.
Hence \eqref{9.7} follows from the strong Markov property
applied to~$T'_x$,
which is a stopping time with respect to a suitable
augmentation of the natural filtration.
\end{pf}

\begin{theorem}\label{H1levy}
Let $A\in\mathcal I$ be a
shift-invariant set.
Then either $\BP_x(A)=0$
for all $x\in\R$ or $\BP_x(A^c)=0$ for all $x\in\R$.
\end{theorem}
\begin{pf}
The proof of Theorem \ref{H1} applies provided
that $\BP_0\{\ell^x\ne0\}=1$ for $\lambda$-a.e. $x\in\R$.
But this follows from \eqref{9.7}.
\end{pf}

Thanks to Theorem \ref{main3} the previous result implies
the following generalisation of Theorem \ref{main2general}.

\begin{theorem}\label{twinslevy}
Let $\mu$ and $\nu$ be orthogonal probability measures on $\R$.
Then the stopping time $T:=T^{\mu,\nu}$ defined
by \eqref{bertoin3} is an unbiased shift under $\BP_\mu$ and
$\BP_\mu\{X_T\in\cdot\}=\nu$.
\end{theorem}

Theorems~\ref{main2a}, \ref{prop3} and \ref{p11} as well as
the minimality properties stated in Theorems~\ref{pminexist} and \ref{pminimal}
do also hold in the present, more general setting.
It would be interesting to study the moment properties
of unbiased shifts of L\'evy processes.
The proof of Theorem \ref{t14complete} makes
significant use of the properties of Brownian motion.
Theorem \ref{posmoment}, however, is still true in the L\'evy case
while the proof of Theorem \ref{pmoment} can be
extended beyond the Brownian case.

\section*{Acknowledgements}
This research started during the Oberwolfach
workshop ``New Perspectives in Stochastic Geometry.''
All support is gratefully
acknowledged. We wish to thank Sergey Foss and Takis Konstantopoulos
for helpful discussions of Theorem~\ref{tpp}, Alex Cox for
providing the idea of the proof of Theorem~\ref{pmoment}, and
Vitali Wachtel for making us aware of Petrov's concentration inequality,
which was used in the proof of Theorem~\ref{t14lower}. We also would
like to thank an
anonymous referee for insightful comments.\looseness=1

%



\printaddresses

\end{document}